\documentclass[11pt]{amsart}

\usepackage{amsmath,amssymb,amsthm,url}
\usepackage{hyperref}
\usepackage{cleveref}

\newtheorem{theorem}{Theorem}[section]
\newtheorem{lemma}[theorem]{Lemma}
\newtheorem{proposition}[theorem]{Proposition}
\newtheorem{corollary}[theorem]{Corollary}

\theoremstyle{definition}
\newtheorem{definition}[theorem]{Definition}

\newtheorem{remark}[theorem]{Remark}

\numberwithin{equation}{section}

\newcommand{\A}{\mathrm{A}}
\newcommand{\C}{\mathrm{C}} 
\newcommand{\D}{\mathrm{D}}
\newcommand{\V}{\mathrm{V}}
\newcommand{\PX}{\mathrm{PX}}

\renewcommand{\L}{\mathcal L}
\renewcommand{\O}{\mathcal O}
\renewcommand{\wr}{\mathop{\mathrm{wr}}}
\newcommand{\ZZ}{\mathbb Z}

\newcommand{\AG}{\mathrm {AG}}
\newcommand{\AAG}{\mathrm {A^2G}}
\newcommand{\HC}{\mathrm {HC}}
\newcommand{\Pet}{F_{10}}
\newcommand{\Hea}{F_{14}}
\newcommand{\Tut}{F_{30}}

\newcommand{\Alt}{\mathop{\mathrm{Alt}}}
\newcommand{\Aut}{\mathop{\mathrm{Aut}}}
\newcommand{\Out}{\mathop{\mathrm{Out}}}
\newcommand{\Sym}{\mathop{\mathrm{Sym}}}
\newcommand{\Gcd}{\mathop{\mathrm{gcd}}}
\newcommand{\lcm}{\mathop{\mathrm{lcm}}}
\newcommand{\ord}{\mathop{\mathrm{ord}}}

\newcommand{\PSL}{\mathop{\mathrm{PSL}}}
\newcommand{\Sp}{\mathop{\mathrm{Sp}}}
\newcommand{\PGL}{\mathop{\mathrm{PGL}}}
\newcommand{\PSU}{\mathop{\mathrm{PSU}}}

\newcommand{\PSp}{\mathop{\mathrm{PSp}}}
\newcommand{\PGammaL}{\mathop{\mathrm{P}\Gamma\mathrm{L}}}
\newcommand{\POmega}{\mathop{\mathrm{P}\Omega}}

\newcommand{\vPX}{\overrightarrow{\PX}}
\newcommand{\vGa}{\overrightarrow{\Gamma}}

\def\Op#1#2{{\bf O}_{{#1}}({{#2}})}
\def\cent#1#2{{\bf C}_{{#1}}({{#2}})}
\def\Zent#1{{\bf Z}({{#1}})}
\def\norm#1#2{{\bf N}_{#1}(#2)}

\title[On the order of vertex-stabilisers]{On the order of vertex-stabilisers in vertex-transitive graphs with local group $\C_p\times\C_p$ or $\C_p \wr \C_2$}

\author{Pablo Spiga}
\address{Pablo Spiga, Dipartimento di Matematica Pura e Applicata, University of Milano-Bicocca, Via Cozzi 53, 20126 Milano, Italy}
\email{pablo.spiga@unimib.it}

\author{Gabriel Verret}
\address{Gabriel Verret, Centre for Mathematics of Symmetry and Computation, The University of Western Australia, 35 Stirling Highway, Crawley, WA 6009, Australia and FAMNIT, University of Primorska, Glagolja\v{s}ka 8, SI-6000 Koper, Slovenia}
\email{gabriel.verret@uwa.edu.au}
\thanks{The second author was supported by UWA as part of the Australian Research Council grant DE130101001.}

\begin{document}


\begin{abstract}
Let $p$ be a prime and let $\L$ be either the intransitive permutation group $\C_p\times\C_p$ of degree $2p$ or the transitive permutation group $\C_p \wr\C_2$ of degree $2p$. Let $\Gamma$ be a connected $G$-vertex-transitive and $G$-edge-transitive graph and let $v$ be a vertex of $\Gamma$. We show that if the permutation group induced by the vertex-stabiliser $G_v$ on the neighbourhood $\Gamma(v)$ is isomorphic to $\L$ then either $|\V(\Gamma)|\geq p|G_v|\log_p\left(|G_v|/2\right)$, or $|\V(\Gamma)|$ is bounded by a constant depending only on $p$, or $\Gamma$ is a very-well understood graph. This generalises a few recent results. 

\end{abstract}
\maketitle

\section{Introduction}
Throughout this paper, all graphs considered will be finite and, unless otherwise specified, simple. A graph $\Gamma$ is said to be $G$-\emph{vertex-transitive} if $G$ is a subgroup of $\Aut(\Gamma)$ acting transitively on the vertex-set $\V(\Gamma)$ of $\Gamma$. Similarly, $\Gamma$ is said to be $G$-\emph{arc-transitive} or $G$-\emph{edge-transitive} if $G$ acts transitively on the arcs or edges of $\Gamma$, respectively. (An \emph{arc} is an ordered pair of adjacent vertices.)

A celebrated theorem of Tutte shows that, if $\Gamma$ is a connected $3$-valent $G$-arc-transitive graph, then the stabiliser $G_v$ of a vertex $v$ in $\Gamma$ has order at most $48$~\cite{Tutte,Tutte2}. This result was later generalised by Trofimov and Weiss who showed that, if $p$ is a prime then there exists a constant $c_p$ depending only on $p$ such that, if $\Gamma$ is a connected $p$-valent $G$-arc-transitive graph and $v$ is a vertex of $\Gamma$ then $|G_v|\leq c_p$~\cite{Trof,Trof2,Wep}. 

The situation is quite different when the valency is not a prime. Indeed, given a composite integer $k$, it is not hard to construct an infinite family of pairs $(\Gamma_i,G_i)$ such that $\Gamma_i$ is a connected $k$-valent $G_i$-arc-transitive graph and $|(G_i)_v|$ grows exponentially with $|\V(\Gamma_i)|$. (See~\cite[Theorem 7]{PSVfRestrictive} for example.)

Recently, Poto\v{c}nik and the authors proved that, when $k=4$, the pairs exhibiting such exponential growth are very special. More precisely, they have shown that there exists a sub-linear function $f$ such that if $\Gamma$ is a connected $4$-valent $G$-arc-transitive graph then either $|G_v|\leq f(|\V(\Gamma)|)$ or $\Gamma$ is part of a well-understood family of graphs~\cite{PSV4valent}.

Our goal is to extend this result to a more general setting. We first need a few definitions.

\begin{definition}
Let $\L$ be a permutation group, let $\Gamma$ be a $G$-vertex-transitive graph and let $G_v^{\Gamma(v)}$ denote the permutation group induced by the action of $G_v$ on the neighbourhood $\Gamma(v)$ of a vertex $v$. Then $(\Gamma,G)$ is said to be \emph{locally-$\L$} if $G_v^{\Gamma(v)}$ is permutation isomorphic to $\L$.
\end{definition}

Throughout this paper, $\C_n$ denotes a cyclic group of order $n$.

\begin{definition}
Let $p$ be a prime. Let $\L_{p,1}$ be the intransitive permutation group $\C_p\times\C_p$ of degree $2p$ and let $\L_{p,2}$ be the transitive permutation group $\C_p \wr\C_2$ of degree $2p$.
\end{definition}

The main result of our paper is the following.

\begin{theorem}\label{theorem:main2}
Let $p$ be a prime and let $\chi\in\{1,2\}$. There exists a function $c$ depending only on $p$ such that, if $(\Gamma,G)$ is a locally-$\L_{p,\chi}$ pair where $\Gamma$ is connected and $G$-edge-transitive and $v$ is a vertex of $\Gamma$, then  one of the following occurs:
\begin{enumerate}
\item $\Gamma\cong \PX(p,r,s)$ for some $r\geq 3$ and $1\leq s\leq r-2$; \label{labelmain1}
\item $|\V(\Gamma)|\geq 2p\frac{|G_v|}{\chi}\log_p\left(\frac{|G_v|}{\chi}\right)$; \label{mainbound}
\item $|\V(\Gamma)|\leq c(p)$. \label{small}
\end{enumerate}
\end{theorem}

The graphs $\PX(p,r,s)$ appearing in Theorem~\ref{theorem:main2} will be described in Section~\ref{sec:graphsCPRS}. We have formulated Theorem~\ref{theorem:main2} in this way for convenience while, in fact, we prove stronger results. Not only we do give an explicit upper bound on $c(p)$ but we prove certain structural results which should help to classify all \emph{exceptional pairs} $(\Gamma,G)$, that is, the pairs satisfying the hypothesis of Theorem~\ref{theorem:main2} but not Theorem~\ref{theorem:main2}~(\ref{labelmain1}) or~(\ref{mainbound}). See Section~\ref{sec:main} for details. It is quite likely that the pairs meeting the inequality in Theorem~\ref{theorem:main2}~(\ref{mainbound}) can be classified (see Remark~\ref{remark:equality}). 

Note that, if $\chi=1$ in Theorem~\ref{theorem:main2} then $G$ is transitive on vertices and edges of $\Gamma$ but not on its arcs. In other words, $\Gamma$ is \emph{$G$-half-arc-transitive.} This implies that $\Gamma$ admits an orientation $\vGa$ as a connected $G$-arc-transitive asymmetric digraph of out-valency $p$ with $G_v^{\vGa^+(v)}\cong\C_p$ (see Lemma~\ref{lemma:graphtodigraph}). 
Conversely, if $\vGa$ is a connected $G$-arc-transitive asymmetric digraph  of out-valency $p$ with $G_v^{\vGa^+(v)}\cong\C_p$ and $|G_v|>p$, and $\Gamma$ is the underlying graph of $\vGa$, then $(\Gamma,G)$ is locally-$\L_{p,1}$ and thus satisfies the hypothesis of Theorem~\ref{theorem:main2}. In particular, Theorem~\ref{theorem:main2} has the following immediate corollary:

\begin{corollary}\label{cor:main}
Let $p$ be a prime. There exists a function $c$ depending only on $p$ such that, if $\vGa$ is a connected $G$-arc-transitive asymmetric digraph of out-valency $p$ with $G_v^{\vGa^+(v)}\cong\C_p$ for some vertex $v$ of $\Gamma$, then one of the following occurs:
\begin{enumerate}
\item $\vGa\cong \vPX(p,r,s)$ for some $r\geq 3$ and $1\leq s\leq r-2$; 
\item $|\V(\vGa)|\geq 2p|G_v|\log_p|G_v|$; \label{mainbound2}
\item $|\V(\vGa)|\leq c(p)$. 
\end{enumerate}
\end{corollary}

We note that Corollary~\ref{cor:main} is an improvement on~\cite[Theorem~1.1]{PotocnikVerret}. Similarly, Theorem~\ref{theorem:main2} with $\chi=2$ is an improvement on~\cite[Theorem A]{Verret2}. The bounds corresponding to Theorem~\ref{theorem:main2}~(\ref{mainbound}) and Corollary~\ref{cor:main}~(\ref{mainbound2}) obtained in those papers were much worse. On the other hand, the proofs did not depend on the Classification of Finite Simple Groups whereas ours does.

There are several reasons why one might like to bound $|G_v|$ in terms of $|\V(\Gamma)|$ in a locally-$\L$ pair $(\Gamma,G)$. For example, if $|G_v|$ can be bounded by a reasonably tame function of $|\V(\Gamma)|$ then the method described in \cite{ConDob} can be applied to obtain a complete list of all locally-$\L$ pairs up to a reasonable order.

The main result of~\cite{PSV4valent} is Theorem~\ref{theorem:main2} in the case $p=\chi=2$. (Note that $\L_{2,2}\cong\D_4$, the dihedral group of order $8$.) Moreover, the corresponding exceptional pairs were also classified in~\cite{PSV4valent} which allowed the construction of a census of all locally-$\D_4$ pairs $(\Gamma,G)$ with $|\V(\Gamma)|\leq 640$. Combined with a previous census of $2$-arc-transitive graphs by Poto\v{c}nik~\cite{Potocnik}, this yielded a census of all $4$-valent arc-transitive graphs of order at most $640$ (see~\cite{PotocnikHomepage}). This in turn also allowed us to obtain a census of all $3$-valent vertex-transitive graphs of order at most $1280$~\cite{PSV1280}. We expect Theorem~\ref{theorem:main2} to bear similar juicy fruits. 

Here is an immediate example. In Section~\ref{sec:(2,1)}, we classify the exceptional pairs corresponding to the case $p=2$, $\chi=1$ in Theorem~\ref{theorem:main2}. This is then used in~\cite{PSVHat} to obtain a census of all $4$-valent graphs of order at most 1000 admitting a half-arc-transitive group of automorphisms. (In fact, all connected arc-transitive asymmetric digraphs of order at most 1000 are obtained and the former is simply a corollary.)

\section{Main theorems and structure of the paper}\label{sec:main}

The hypothesis of Theorem~\ref{theorem:main2} will be used repeatedly. For brevity's sake, we will call it Hypothesis~A.

\medskip

\noindent\textbf{Hypothesis~A.} Let $p$ be a prime, let $\chi\in\{1,2\}$, let $(\Gamma,G)$ be a locally-$\L_{p,\chi}$ pair such that $\Gamma$ is connected and $G$-edge-transitive and let $v$ be a vertex of $\Gamma$.

\medskip

All of the results of this section are proved in Sections~\ref{sec:mainproofs} and~\ref{sec:semisimple}. Our main tool is the following result. (A permutation group is called \emph{semiregular} if all its point-stabilisers are trivial. Definitions and results concerning quotient graphs $\Gamma/N$ and regular covers can be found in Section~\ref{sec:quotient}.)

\begin{theorem}\label{theorem:main1}
Assume Hypothesis~A. Then one of the following occurs:
\begin{enumerate}
\item $\Gamma\cong \PX(p,r,s)$ for some $r\geq 3$ and $1\leq s\leq r-2$; \label{PXU}
\item $|\V(\Gamma)|\geq 2p\frac{|G_v|}{\chi}\log_p\left(\frac{|G_v|}{\chi}\right)$; \label{bound}
\item $G$ has a semiregular abelian minimal normal subgroup having at most two orbits; \label{twoorbits}
\item $G$ has a semiregular abelian minimal normal subgroup $N$ such that $\Gamma/N$ is a cycle of length at least $3$  and $(\ref{PXU})$ does not hold; \label{cycle}
\item $G$ has a unique minimal normal subgroup and this subgroup is non-abelian; \label{semisimple}
\item $G$ has a non-identity normal subgroup $N$ such that $\Gamma$ is a regular cover of $\Gamma/N$ and one of $(\ref{twoorbits})$, $(\ref{cycle})$, or $(\ref{semisimple})$ is satisfied with $(\Gamma,G)$ replaced by $(\Gamma/N,G/N)$. Moreover $N$ is soluble unless $(\ref{semisimple})$ is satisfied (with $(\Gamma,G)$ replaced by $(\Gamma/N,G/N)$).\label{funny}
\end{enumerate}
\end{theorem}

Theorem~\ref{theorem:main1}~(\ref{PXU}) and~(\ref{bound}) neatly correspond to Theorem~\ref{theorem:main2}~(\ref{labelmain1}) and~(\ref{mainbound}). To obtain Theorem~\ref{theorem:main2} from Theorem~\ref{theorem:main1}, it thus suffices to obtain a bound on $|\V(\Gamma)|$ or $|G_v|$ in each of the remaining cases of Theorem~\ref{theorem:main1}. This is exactly what we do in Theorems~\ref{thm:mainbasic},~\ref{GarPraCycle} and~\ref{theo:semisimple} for cases~(\ref{twoorbits}), (\ref{cycle}) and (\ref{semisimple}), respectively. These are in some sense the basic cases with respect to case~(\ref{funny}) and the result in this last case follows immediately from the theory of regular covers (see Section~\ref{sec:quotient}).

\begin{theorem}\label{thm:mainbasic}
Assume Hypothesis~A. If $G$ has a semiregular abelian normal subgroup having at most two orbits then $G_v$ acts faithfully on $\Gamma(v)$ and, in particular, $|G_v|=\chi p^2$.
\end{theorem}

If $p$ and $q$ are distinct primes, we denote by $\ord_p(q)$ the smallest positive integer $\ell$ with $q^\ell\equiv 1\mod p$.

\begin{theorem}\label{GarPraCycle}
Assume Hypothesis~A and assume that $G$ has a semiregular abelian minimal normal $q$-subgroup $N$ such that $\Gamma/N$ is a cycle of length $m\geq 3$. Let $K$ be the kernel of the action of $G$ on the $N$-orbits. Then $|G_v|=\chi p^t$, $|G_v:K_v|=\chi$ and one of the following occurs:
\begin{enumerate}
\item $\Gamma\cong \PX(p,r,s)$ for some $r\geq 3$ and $1\leq s\leq r-2$; \label{first}
\item $t\leq m$, $p\neq q$, $K_v$ is an elementary abelian $p$-group acting faithfully by conjugation on $N$, and $|\V(\Gamma)|\geq mq^{t\ord_p(q)}$. \label{second}
\end{enumerate}
\end{theorem}

\begin{theorem}\label{theo:semisimple}
Assume Hypothesis~A. If $G$ has a unique minimal normal subgroup $N$ and $N$ is non-abelian then either $|\V(\Gamma)|> 2p\frac{|G_v|}{\chi}\log_p\left(\frac{|G_v|}{\chi}\right)$ or $(p,\chi,|\V(\Gamma)|,|G_v|,N,G)$ appears in Table~\ref{mutherfucker}. In particular, $p\leq 13$ and 
$$\log_p\left(\frac{|G_v|}{\chi}\right) \leq \begin{cases} 8 &\mbox{if } p=2, \\
6 & \mbox{if } p=3, \\
3 & \mbox{if } p\in \{5,7,13\}, \\
2 & \mbox{if } p=11. \\
\end{cases}$$
\end{theorem}

\begin{table}[!h]
\begin{tabular}{|c|c|c|c|c|c|c|c|c|c|}\hline
$p$&$\chi$&$|\V(\Gamma)|$&$|G_v|$&$N$&$G$&Comments\\\hline
$2$&$1$&$30$&$2^2$&$\Alt(5)$&$\Sym(5)$&\\
$2$&$1$&$90$&$2^3$ or $2^4$&$\PSL_2(9)$&$|\PGammaL_2(9):G|\leq 2$&$G\neq \Sym(6)$\\
$2$&$1$&$8100$&$2^8$&$\PSL_2(9)^2$&$|\PGammaL_2(9)\wr\Sym(2):G|=2$&$G\neq \PGammaL_2(9)^2$\\
$3$&$1$&$11648$&$3^6$&$G_2(3)$&$\Aut(N)$&\\
$2,3,7$&$1$&$\frac{2(p^3-1)(p^2-1)}{\Gcd(3,p-1)}$&$p^3$&$\PSL_3(p)$&$N\rtimes\langle\iota\rangle$&$\iota$ graph aut.\\
$5,7$&$1$&$\frac{(p^2-1)^2}{2}$&$p^2$&$\PSL_2(p)^2$&$\PSL_2(p)\wr\Sym(2)$&\\

\hline

$2$&$2$&$15$&$2^3$&$\Alt(5)$&$\Sym(5)$&\\
$2$&$2$&$45$&$2^4$ or $2^5$&$\PSL_2(9)$&$|\PGammaL_2(9):G|\leq 2$& $G\neq \Sym(6)$\\
$2$&$2$&$90$&$2^4$&$\PSL_2(9)$&$\PGammaL_2(9)$&\\
$2$&$2$&$8100$&$2^9$&$\PSL_2(9)^2$&$\PGammaL_2(9)\wr\Sym(2)$&\\
$3$&$2$&$5824$&$2\cdot 3^6$&$G_2(3)$&$\Aut(N)$&\\
$7$&$2$&$1152$&$2\cdot 7^2$&$\PSL_2(7)^2$&$|\PGL_2(7)\wr\Sym(2):G|=2$&$G\neq \PGL_2(7)^2$\\
   &   &      &            &             &               & $G/N$ not cyclic\\
$2,3,5,7,13$&$2$&$\frac{(p^3-1)(p^2-1)}{\Gcd(3,p-1)}$&$2p^3$&$\PSL_3(p)$&$N\rtimes\langle\iota\rangle$&$\iota$ graph aut.\\
$5,7,11,13$&$2$&$\frac{(p^2-1)^2}{4}$&$2p^2$&$\PSL_2(p)^2$&$\PSL_2(p)\wr\Sym(2)$&\\
\hline

\end{tabular}
\caption{Exceptional pairs for Theorem~\ref{theo:semisimple}}\label{mutherfucker}
\end{table}
Together, Theorems~\ref{theorem:main1},~\ref{thm:mainbasic},~\ref{GarPraCycle} and~\ref{theo:semisimple} easily imply Theorem~\ref{theorem:main2}. In fact, they also yield an explicit bound on $c(p)$. 

\begin{proposition}\label{prop:explicit}
Let $e(p)=\frac{\log(2p)}{\log(1+1/p)}$. In Theorem~$\ref{theorem:main2}$, we can take $c(p) = \begin{cases} 8100 &\mbox{if } p=2, \\
2e(p)p^{e(p)+1} & \mbox{if } p\geq 3. \end{cases}$
\end{proposition}

The results of this section suggest a strategy for dealing with the exceptional pairs in Theorem~\ref{theorem:main2} for fixed $(p,\chi)$. Namely, first  find the exceptional pairs corresponding to each of Theorems~\ref{thm:mainbasic},~\ref{GarPraCycle} and~\ref{theo:semisimple} and then find all the regular covers that are still exceptional. For any fixed $p$, this is a ``finite problem'' but the difficulty increases with $p$, so it is not clear if this can be done in full generality. See Section~\ref{sec:(2,1)} where this procedure is implemented in the case $(p,\chi)=(2,1)$.

We now give a brief outline of the rest of the paper. Section~\ref{sec:quotient} contains some basic definitions and lemmas that are needed for the rest of the paper. In Section~\ref{locally-L} we study locally-$\L_{p,\chi}$ pairs and prove more preliminary results. The graphs $\PX(p,r,s)$ are defined in Section~\ref{sec:graphsCPRS} where we also establish some useful results concerning them. These are then used in Section~\ref{sec:nc} to deal with the crucial case of regular covers of $\PX(p,r,s)$. By Section~\ref{sec:mainproofs}, we are ready to prove Theorems~\ref{theorem:main1},~\ref{thm:mainbasic} and~\ref{GarPraCycle} and Proposition~\ref{prop:explicit}. The proof of Theorem~\ref{theo:semisimple} requires a detailed case-by-case analysis using the Classification of the Finite Simple Groups and is delayed until Section~\ref{sec:semisimple}.

\section{Preliminaries}\label{sec:quotient}

In this section, we collect various small technical results which will be used in later sections.

\subsection{Basic inequalities}

\begin{lemma}\label{lemma:stupidtechnical}
For $t$ and $p$ positive integers, we have $(p+1)^{t+1}> (p+t+1)p^t$. 
\end{lemma}
\begin{proof}
By Bernoulli's inequality, we have $(1+1/p)^t\geq 1+t/p$. Therefore $(p+1)^t\geq (1+t/p)p^t$ and hence $(p+1)^{t+1}\geq (p+1)(1+t/p)p^t=(p+t+1+t/p)p^t>(p+t+1)p^t$.
\end{proof}


\subsection{Quotient graphs}\label{subsec:quotient}
Given a graph $\Gamma$ and a group $N\leq\Aut(\Gamma)$, the {\em quotient graph} $\Gamma/N$ is the graph whose vertices are the $N$-orbits, and with  two such $N$-orbits $v^N$ and $u^N$ adjacent whenever there is a pair of vertices $v'\in v^N$ and $u'\in u^N$ that are adjacent in $\Gamma$. If the natural projection $\pi:\Gamma\to \Gamma/N$ is a local bijection (that is, if $\pi_{|\Gamma(v)}:\Gamma(v)\to (\Gamma/N)(v^N)$  is a bijection for every $v\in\V(\Gamma)$) then $\Gamma$ is called a \emph{regular cover} of $\Gamma/N$.

We will need a few facts about quotient graphs which we collect in the following lemma (these are mostly folklore, see for example \cite[Section~1]{P85}).

\begin{lemma}\label{lemma:quotients}
Let $\Gamma$ be a connected $G$-vertex-transitive graph, let $v$ be a vertex of $\Gamma$, let $N$ be a normal subgroup of $G$ and let $K$ be the kernel of the action of $G$ on $N$-orbits. Then
\begin{enumerate}
\item $\Gamma/N$ is connected;\label{a1}
\item $G/K$ acts faithfully as a group of automorphisms of $\Gamma/N$;\label{a2}
\item $\Gamma/N$ is $G/K$-vertex-transitive;\label{a3}
\item the valency of  $\Gamma/N$ is at most the number of orbits of $K_v^{\Gamma(v)}$;\label{a5}
\item if $\Gamma$ is $G$-edge-transitive then $\Gamma/N$ is $G/K$-edge-transitive.\label{a4}

\smallskip
\hspace{-1.85cm} Moreover, if $\Gamma$ is a regular cover of $\Gamma/N$ then 
\smallskip

\item the valency of  $\Gamma/N$ is equal to the valency of $\Gamma$; \label{a6}
\item $N$ is semiregular;\label{a7}
\item $N=K$;\label{a8}
\item $G_v\cong (G/N)_{v^N}$;\label{a9}
\item $G_v^{\Gamma(v)}\cong (G/N)_{v^N}^{(\Gamma/N)(v^N)}$.\label{a10}
\end{enumerate}
\end{lemma}

As an immediate corollary of Lemma~\ref{lemma:quotients}, we have:

\begin{corollary}\label{lemma:cover}
Assume Hypothesis~A. If $N$ is a normal subgroup of $G$ such that $\Gamma$ is a regular cover of $\Gamma/N$ then $(\Gamma/N,G/N)$ is a locally-$\L_{p,\chi}$ pair such that $\Gamma/N$ is connected and $G/N$-edge-transitive. In particular, Hypothesis~A holds with $(\Gamma,G)$ replaced by $(\Gamma/N,G/N)$.
\end{corollary}

\subsection{$p$-groups}\label{sub:pgrp}
Recall that  $\ord_p(q)$ denotes the smallest positive integer $\ell$ with $q^{\ell}\equiv 1\mod p$.

\begin{lemma}\label{PabloLemma} Let $p$ and $q$ be distinct primes and let $H$ be an elementary abelian $p$-subgroup of $\mathrm{GL}_n(q)$ of order $p^t$. Then $n\geq t\ord_p(q)$ and $q^n\geq (p+1)^t$.

\end{lemma}

\begin{proof}
Let $V$ be the $n$-dimensional vector space of column vectors over the finite field $\mathbb{F}_q$. The action of $H$ on this vector space allows us to view $V$ as an $\mathbb{F}_qH$-module. Let $\cent V H$ be the centraliser of $H$ in $V$ and let $W=\langle -v+vh\mid v\in V,h\in H\rangle$. As $|H|$ is coprime to $q$, it follows from~\cite[$8.2.7$]{KS} that $V=\cent V H\oplus W$ and from Maschke's theorem that $W=W_1\oplus W_2\oplus \cdots \oplus W_\ell$, with $W_i$ an irreducible $\mathbb{F}_qH$-module for every $i\in \{1,\ldots,\ell\}$. Observe that $W_i$ is a non-trivial $\mathbb{F}_qH$-module because $W_i\cap \cent V H=0$.

For every $i\in \{1,\ldots,\ell\}$ let $C_i=\cent H{W_i}$ and observe that $\bigcap_{i=1}^\ell C_i=1$ because $H$ acts faithfully on $V$ and hence on $W$. Since $H$ is elementary abelian, it follows from~\cite[$9.4.3$]{DR} that $\dim_{\mathbb{F}_q}W_i=\ord_p(q)$ and $|H:C_i|=p$. In particular, $\ell\geq t$ and $n\geq\dim_{\mathbb{F}_q}W=\ell\ord_p(q)\geq t\ord_p(q)$, which concludes the first part of the proof. 

Since $q^{\ord_p(q)}\geq p+1$, we also have  $q^n\geq q^{t\ord_p(q)}\geq (p+1)^t$.
\end{proof}

\begin{lemma}\label{Pablolemma2}
Let $P=\langle x_0,\ldots,x_n\rangle$ be a $p$-group. If,  for some $i\in\{0,\ldots,n-1\}$ and $g\in P$, we have $x_n=x_i^g$, then $P=\langle x_0,\ldots,x_{n-1}\rangle$.   
\end{lemma}

\begin{proof}
We recall that $g$ is called a non-generator of $P$ if, for every subset $X$ of $P$, $P=\langle g,X\rangle$ implies that $P=\langle X\rangle$. In a $p$-group, every commutator is a non-generator (see~\cite[$5.3.2$]{DR}). Assume $x_n=x_i^g$ for some $i\in\{0,\ldots,n-1\}$ and $g\in P$. We have
\begin{eqnarray*}
P&=&\langle x_0,\ldots,x_{n-1},x_n\rangle=\langle
x_0,\ldots,x_{n-1},x_i^g\rangle=\langle x_0,\ldots,x_{n-1},x_i[x_i,g]\rangle\\ 
&=&\langle x_0,\ldots,x_{n-1},[x_i,g]\rangle=\langle x_0,\ldots,x_{n-1}\rangle.
\end{eqnarray*}
\end{proof}

As usual, we denote by $(\gamma_i(P))_i$ the lower central series of the group $P$.
\begin{lemma}\label{poddtechnical}
Let $P$ be a $p$-group generated by elements of order $p$ and such that $\gamma_2(P)$ is elementary abelian. Then $P$ has exponent at most $p^2$.
\end{lemma}
\begin{proof}
Let $x$ and $y$ be two elements of $P$. Note that $(xy)^p\equiv x^py^p\mod \gamma_2(P)$.  Using induction and the fact that $P$ is generated by elements of order $p$, this shows that $z^p\in \gamma_2(P)$ for every $z\in P$. As $\gamma_2(P)$ is elementary abelian, $P$ has exponent at most $p^2$.
\end{proof}

\section{Locally-$\L_{p,\chi}$ pairs}\label{locally-L}
In this section, we prove a few facts about locally-$\L_{p,\chi}$ pairs. Recall that a \emph{block} of a permutation group $G$ on the set $\Omega$ is a subset $B$ of $\Omega$ such that for every $g\in G$, we have $B^g=B$ or $B^g\cap B=\emptyset$.  A block is called \emph{non-trivial} if $1< |B|<|\Omega|$ (and \emph{trivial} otherwise).

\begin{lemma}\label{lemma:blocks}
Let $p$ be a prime and let $\chi\in\{1,2\}$. 
\begin{enumerate}
\item $\L_{p,1}$ is the unique intransitive subgroup of index $\chi$ of $\L_{p,\chi}$. \label{unique}
\item The only non-trivial blocks of $\L_{p,\chi}$ are the orbits of $\L_{p,1}$. \label{blocks}
\end{enumerate}
\end{lemma}
\begin{proof}
We first show (\ref{unique}). If $\chi=1$, there is nothing to show hence we assume that $\chi=2$. If $p\neq 2$, then $\L_{p,1}$ is a normal Sylow $p$-subgroup of $\L_{p,2}$ and hence it is the unique subgroup of index $2$ in $\L_{p,2}$. If $p=2$, then $\L_{p,2}\cong\D_4$ and the statement is routine.

We now show (\ref{blocks}). Let $B$ be a block of $\L_{p,\chi}$ and let $\O_1$ and $\O_2$ be the orbits of $\L_{p,1}$. Note that, for $i\in\{1,2\}$, $\L_{p,1}$ acts primitively on $\O_i$ and hence $B\cap \O_i$ is either empty, a singleton, or $\O_i$. Suppose that $B\cap \O_1$ is a singleton $\{\omega\}$ while $B\cap\O_2\neq\emptyset$. There exists a non-identity $g\in \L_{p,1}$ fixing $\O_2$ pointwise but not fixing $\omega$. This shows that $B$ is not a block of $\L_{p,1}$, which is a contradiction. By reversing the roles played by $\O_1$ and $\O_2$, we have proved that $B$ must either be trivial, $\O_1$ or $\O_2$. As $\L_{p,1}\unlhd \L_{p,\chi}$, $\O_1$ and $\O_2$ are in fact blocks of $\L_{p,\chi}$.
\end{proof}

A \emph{digraph} $\vGa$ consists of a non-empty set of \emph{vertices} $\V(\vGa)$ and a set of \emph{arcs} $\A(\vGa)\subseteq \V(\vGa)\times \V(\vGa)$, which is an arbitrary binary relation on $V$. A graph is then simply a digraph $\vGa$ such that $\A(\vGa)$ is symmetric. A digraph $\vGa$ is called {\em asymmetric} provided that the relation $\A(\vGa)$ is asymmetric.  

An automorphism of a digraph $\vGa$ is a permutation of $\V(\vGa)$ which preserves the relation $\A(\vGa)$. We say that $\vGa$ is $G$-arc-transitive provided that $G$ is a group of automorphisms of $\vGa$ acting transitively on $\A(\vGa)$. 

If $(u,v)$ is an arc of $\vGa$ then we say that $v$ is an {\em out-neighbour} of $u$. The symbols $\vGa^+(v)$ will denote the set of out-neighbors of $v$. The digraph $\vGa$ is said to be of out-valency $k$ if $|\vGa^+(v)| = k$ for every $v\in \V(\vGa)$. Given a graph $\Gamma$ an \emph{asymmetric orientation} of $\Gamma$ is a digraph with the same vertex-set as $\Gamma$ and having exactly one arc out of every inverse pair of arcs of $\Gamma$.

\begin{lemma}\label{lemma:graphtodigraph}
Assume Hypothesis~A. If $\chi=1$ then $\Gamma$ has a $G$-arc-transitive asymmetric orientation $\vGa$ of out-valency $p$ with $G_v^{\vGa^+(v)}\cong\C_p$.
\end{lemma}
\begin{proof}
By hypothesis, $\Gamma$ is $G$-vertex-transitive and $G$-edge-transitive. Since $\L_{p,1}$ is not transitive, $\Gamma$ is not $G$-arc-transitive and is thus $G$-half-arc-transitive. It follows that $G$ has exactly two orbits on arcs of $\Gamma$ and every arc is in a different orbit than its inverse.  Let $a$ be an arc of $\Gamma$ and let $\vGa$ be the digraph with vertex-set $\V(\Gamma)$ and arc-set $\{a^g\mid g\in G\}$. Clearly, $\vGa$ is a $G$-arc-transitive asymmetric orientation of $\Gamma$ of out-valency $p$ with $G_v^{\vGa^+(v)}\cong\C_p$.
\end{proof}

\begin{lemma}\label{lemma:Gv*}
Assume Hypothesis~A. There exists a unique subgroup $G_v^*$ of index $\chi$ in $G_v$ such that $(G_v^*)^{\Gamma(v)}\cong\L_{p,1}$. Moreover, $G_v^*$ is a $p$-group.
\end{lemma}
\begin{proof}
By Lemma~\ref{lemma:blocks} (\ref{unique}), $\L_{p,1}$  is the unique intransitive subgroup of index $\chi$ in $\L_{p,\chi}$. It follows that there is a unique subgroup $G_v^*$ of index $\chi$ in $G_v$ such that $(G_v^*)^{\Gamma(v)}\cong\L_{p,1}$.

Let $u$ be a neighbour of $v$ and let $G_v^{[1]}$ denote the kernel of the action of $G_v$ on $\Gamma(v)$. Since every point-stabiliser of $\L_{p,\chi}$ has order $p$, it follows that $|G_{uv}:G_v^{[1]}|=p$. A standard argument yields that $G_{uv}$ is a $p$-group: suppose that $G_{uv}$ contains an element $g$ of order coprime to $p$. Since $g$ is non-trivial, it must move some vertex. Let $w$ be a vertex of $\Gamma$ moved by $g$ at minimal distance from the arc $(u,v)$. By the connectivity of $\Gamma$ and the choice of $w$, there is a path $u_0,u_1,\ldots,u_t,w$ such that $\{u_0,u_1\}=\{u,v\}$ and $g$ fixes each $u_i$. Then $g\in G_{u_{t-1},u_t}$ and $g$ acts nontrivially on $\Gamma(u_t)$. This contradicts the fact that $G_{u_{t-1},u_t}^{\Gamma(u_t)}$ is a $p$-group. 

We have shown that $G_{uv}$ is a $p$-group. As $|G_v^*:G_{uv}|=p$, the result follows.
\end{proof}

From now on, when Hypothesis~A holds, we treat Lemma~\ref{lemma:Gv*} as the definition of $G_v^*$.

\begin{lemma}\label{lemma:easy}
Assume Hypothesis~A. If $N$ is a normal subgroup of $G$ then one of the following occurs:
\begin{enumerate}
\item $N$ has at most two orbits;
\item $\Gamma/N$ is a cycle of length at least $3$; \label{cyclecase}
\item $\Gamma$ is a regular cover of $\Gamma/N$. \label{reduction}
\end{enumerate}
\end{lemma}
\begin{proof}
By Lemma~\ref{lemma:quotients}, $\Gamma/N$ is a connected vertex-transitive graph. Let $k$ be the valency of $\Gamma/N$. If $k\leq 1$ then~(1)  holds. If $k=2$ then~(2) holds. We may thus assume that $k>2$. Note that the equivalence classes for the equivalence relation ``being in the same $N$-orbit'' on $\Gamma(v)$ are blocks of $G_v^{\Gamma(v)}\cong\L_{p,\chi}$ and the number of such equivalence classes is exactly $k$.   By Lemma~\ref{lemma:blocks} (\ref{blocks}), it follows that $k=2p$ and $\Gamma$ is a regular cover of $\Gamma/N$.
\end{proof}

We now study Lemma~\ref{lemma:easy}~(\ref{cyclecase})  in more detail.

\begin{lemma}\label{lemma:cyclecycle}
Assume Hypothesis~A. Let $K$ be a normal subgroup of $G$ such that $\Gamma/K$ is a cycle of length $m\geq 3$ and $G/K$ acts faithfully on $\Gamma/K$.  Then an asymmetric orientation of $\Gamma/K$ induces an asymmetric orientation of $\Gamma$, yielding an asymmetric digraph $\vGa$ of in- and out-valency $p$, whose underlying graph is $\Gamma$. Let $G^+/K$ be the orientation-preserving subgroup of $G/K$. Then $G_v^*=(G^+)_v=K_v$, $\vGa$ is $G^+$-vertex-transitive, $G^+/K\cong\C_m$ and $(K_v)^{\vGa^+(v)}\cong\C_p$.
\end{lemma}
\begin{proof}
Fix an asymmetric orientation of the cycle $\Gamma/K$. This induces an asymmetric orientation of $\Gamma$, yielding an asymmetric digraph $\vGa$ whose underlying graph is $\Gamma$. Clearly $\vGa$ (and thus also $\Gamma$) is $G^+$-vertex-transitive and, since $\Gamma$ is $2p$-valent, $\vGa$ has in- and out-valency $p$. Moreover, $G^+/K$ is a vertex-transitive group of automorphisms of the directed cycle $\vGa/K$ and thus $G^+/K\cong\C_m$. Further, $G^+/K<G/K$ if and only $G/K$ is arc-transitive, which in turn happens if and only if $\chi=2$. In particular, $|G/K:G^+/K|=\chi$. Since $G^+$ is vertex-transitive, it follows that $|G_v:(G^+)_v|=|G:G^+|=|G/K:G^+/K|=\chi$.

Note that $(G_v^+)^{\Gamma(v)}$ has two orbits which induce a system of imprimitivity for $G_v^{\Gamma(v)}$ with two blocks of size $p$. By Lemma~\ref{lemma:blocks}, it follows that $(G_v^+)^{\Gamma(v)}\cong\L_{p,1}$ and $(G^+_v)^{\vGa^+(v)}\cong\C_p$.  Since $|G_v:G^+_v|=\chi$, Lemma~\ref{lemma:Gv*} then implies $G_v^*=G_v^+$. Clearly, $G_v^+=K_v$.
\end{proof}

Let $\Gamma$ be a graph and let $N\leq\Aut(\Gamma)$. Besides the quotient graph $\Gamma/N$ defined in Section~\ref{subsec:quotient}, there is another way to quotient $\Gamma$ by $N$ which is sometimes more useful. To distinguish it from $\Gamma/N$, we call this other quotient the multi-quotient of $\Gamma$ with respect to $N$.  The precise definition of this quotient requires some setup and can be found in~\cite[Section~3]{CoveringReference}, for example. 

Since we only use it in the proof of Lemma~\ref{lemma:Malnic} and a precise definition would take us too far astray, we simply give a (slightly incomplete) definition. The \emph{multi-quotient} of $\Gamma$ with respect to $N$ has as vertices the orbits of $N$ on the vertices of $\Gamma$ and as arcs the orbits of $N$ on the arcs of $\Gamma$. Incidence is defined in the most natural way: if $(u,v)$ is an arc of $\Gamma$ then the arc $(u,v)^N$ of the multi-quotient has initial vertex $u^N$ and terminal vertex $v^N$. Note that, even if $\Gamma$ is a simple graph, the multi-quotient may not be: according to the terminology of~\cite[Section~3]{CoveringReference} it may have multiple edges, loops and semi-edges.

\begin{lemma}\label{lemma:Malnic}
Assume Hypothesis~A and assume that $G$ has a semiregular normal subgroup $N$ such that $\Gamma/N$ is a cycle of length $m\geq 3$. Let $K$ be the kernel of the action of $G$ on the $N$-orbits. Then $K_v$ is an elementary abelian $p$-group of order at most $p^m$.
\end{lemma}
\begin{proof}

Let $\Gamma'$ be the multi-quotient of $\Gamma$ with respect to $N$. Since $N$ is semiregular, $\Gamma$ is a regular cover of $\Gamma'$ and $G/N$ acts faithfully on $\Gamma'$ (in the sense of~\cite[Section~3]{CoveringReference}). In particular, $\Gamma'$ has the same valency as $\Gamma$, namely $2p$. Since $\Gamma'$ is edge-transitive and its underlying simple graph is a cycle of length $m$, $\Gamma'$ is a multi-cycle with $m$ vertices and with $p$ edges between each adjacent pair of vertices. Let $K'$ be the subgroup of $\Aut(\Gamma')$ fixing every vertex of $\Gamma'$. Note that $K'$ is isomorphic to the direct product of $m$ copies of $\Sym(p)$. In particular, a Sylow $p$-subgroup of $K'$ is elementary abelian of order $p^m$. Since $N$ is semiregular, $K_v\cong K_v/(N\cap K_v)\cong NK_v/N= K/N\leq K'$ and thus $K_v$ is isomorphic to a subgroup of $K'$ and the result follows.
\end{proof}

\begin{lemma}\label{lemma:reductionToUnique}
Assume Hypothesis~A. If $G$ has two distinct minimal normal subgroups $M$ and $N$ with $M$ non-abelian then $\Gamma$ is a regular cover of $\Gamma/N$. Moreover, $MN/N$ is a minimal normal subgroup of $G/N$ isomorphic to $M$.
\end{lemma}
\begin{proof}
Note that $MN=N\times M$. Let $K$ be the kernel of the action of $G$ on the $N$-orbits.

Suppose that $M\leq K$. Then $v^M\subseteq v^K=v^N$. Since $M$ is not soluble, by Burnside's Theorem there exists a prime $q$ dividing $|M|$ and different from $2$ and $p$. Let $m\in M$ be an element of order $q$. We have $v^m\in v^N$ and hence $v^m=v^n$ for some $n\in N$. This gives $mn^{-1}\in G_v$. Since $|mn^{-1}|=\lcm(|m|,|n|)=\lcm(q,|n|)$ and $G_v$ is a $\{2,p\}$-group, this is a contradiction.

We may thus assume that $M\nleq K$. By minimality of $M$, we obtain $M\cap K=1$ and hence $M\cong MK/K\leq\Aut(\Gamma/N)$. Suppose that $\Gamma$ is not a regular cover of $\Gamma/N$. Then by Lemma~\ref{lemma:easy}, $\Gamma/N$ either has at most two vertices or is a cycle. In particular, $\Aut(\Gamma/N)$ is soluble which contradicts the fact that it contains a subgroup isomorphic to $M$. This shows that $\Gamma$ is a regular cover of $\Gamma/N$ and, by Lemma~\ref{lemma:quotients}, $N=K$. Clearly, $MN/N$ is a minimal normal subgroup of $G/N$.
\end{proof}

An easy application of Lemma~\ref{lemma:reductionToUnique} yields the following:

\begin{corollary}\label{corollary:nonabelian}
Assume Hypothesis~A. If $G$ has a non-abelian minimal normal subgroup then $G$ contains a normal subgroup $N$ such that $\Gamma$ is a regular cover of $\Gamma/N$ and $G/N$ has a unique minimal normal subgroup and this subgroup is non-abelian.
\end{corollary}
\begin{proof}
Let $U$ be a non-abelian minimal normal subgroup of $G$. We argue by induction on $|\V(\Gamma)|$. If $U$ is the unique minimal normal subgroup of $G$ then we may take $N=1$. Otherwise, let $V$ be a minimal normal subgroup of $G$ with $V\neq U$. By Lemma~\ref{lemma:reductionToUnique}, $\Gamma$ is a regular cover of $\Gamma/V$ and, by Corollary~\ref{lemma:cover}, Hypothesis~A holds with $(\Gamma,G)$ replaced by $(\Gamma/V,G/V)$. Clearly, $UV/V$ is a minimal normal subgroup of $G/V$ isomorphic to $U$. 

Since $|\V(\Gamma/V)|<|\V(\Gamma)|$, we may use induction to conclude that $G/V$ contains a normal subgroup $N/V$ such that $\Gamma/V$ is a regular cover of $(\Gamma/V)/(N/V)\cong \Gamma/N$ and $(G/V)/(N/V)\cong G/N$ has a unique minimal normal subgroup and this subgroup is non-abelian. Since $\Gamma$ is a regular cover of $\Gamma/V$ which is a regular cover of $\Gamma/N$, $\Gamma$ is a regular cover of $\Gamma/N$.
\end{proof}

All the results in this section so far have been direct consequences of the definition of $\L_{p,\chi}$. We now apply some slightly deeper group theoretic results to establish some structural properties of $G_v^*$. For ease of exposition, we mostly follow~\cite{Currano} but these results can actually trace their lineage through Glauberman~\cite{Glaub} and Sims~\cite{Sims} to Tutte's results on $3$-valent arc-transitive graphs~\cite{Tutte,Tutte2}.

\begin{theorem}\label{lemma:ea}
Assume Hypothesis~A. Then the following hold:

\begin{enumerate}
\item $G_v^*$ has nilpotency class at most $3$;\label{bonbonbon}
\item $G_v^*$ contains an elementary abelian $p$-subgroup of order at least  $|G_v^*|^{2/3}$; \label{bonbon}
\item $|\Zent {G_v^*}|^3\geq |G_v^*|$; \label{boubou}
\item $G_v^*$ has exponent at most $p^2$.
\end{enumerate}
\end{theorem}
\begin{proof}

By Lemma~\ref{lemma:Gv*}, $(G_v^*)^{\Gamma(v)}\cong\L_{p,1}$ and $G_v^*$ is a $p$-group. Let $|G_v^*|=p^t$ and let $u$ and $w$ be representatives for the two orbits of $(G_v^*)^{\Gamma(v)}$. We show that the arcs $(u,v)$ and $(v,w)$ are in the same $G$-orbit. We argue by contradiction and we suppose that this is not the case. Since $\Gamma$ is $G$-edge-transitive, it follows that $(u,v)$ is in the same $G$-orbit as $(w,v)$. This implies that $u$ and $w$ are in the same $G_v$-orbit and hence $G_v^{\Gamma(v)}$ is transitive. It follows that $\Gamma$ is $G$-arc-transitive and hence $(u,v)$ and $(v,w)$ are in the same $G$-orbit, which is a contradiction. 

Let $\phi\in G$ such that $(u,v)^\phi=(v,w)$. We show that $\langle G_v^*,\phi \rangle$ is transitive on $\V(\Gamma)$. For $i\in\ZZ$, let $v_i=v^{\phi^i}$. Note that $(v_{-1},v_0,v_1)=(u,v,w)$ and hence $\Gamma(v_0)=(v_{-1})^{G_{v_0}^*}\cup (v_{1})^{G_{v_0}^*}$. Conjugating by $\phi^i$, we obtain that $\Gamma(v_i)=(v_{i-1})^{G_{v_i}^*}\cup (v_{i+1})^{G_{v_i}^*}$ for every $i\in\ZZ$. Let $G^*=\langle G_{v_i}^*\mid i\in\ZZ\rangle$ and let $S=v^{\langle \phi\rangle}=\{v_i\mid i\in\ZZ\}$. Note that $G^*\leq \langle G_v^*,\phi \rangle$, and  hence it suffices to show that $S^{G^*}=\V(\Gamma)$. By contradiction, suppose that there exists a vertex not in $S^{G^*}$ and choose one with minimum distance to $S$. Call this vertex $\alpha$ and let $(p_0,\ldots,p_{n-1},p_n)$ be a shortest path from $\alpha$ to a vertex of $S$. In particular, $p_0=\alpha$ and $p_n=v_i$ for some $i\in\ZZ$.  Since $\Gamma(v_i)=(v_{i-1})^{G_{v_i}^*}\cup (v_{i+1})^{G_{v_i}^*}$, there exists $\sigma\in G_{v_i}^*\leq G^*$ such that $(p_{n-1})^\sigma\in\{v_{i-1},v_{i+1}\}\subseteq S$. Since $\alpha$ is not in $S^{G^*}$, neither is $\alpha^\sigma$, but $\alpha^\sigma$ is closer to $S$ than $\alpha$ is, which is a contradiction. 

From now on, we follow the notation of~\cite{Currano} and~\cite{Glaub} as closely as possible. Let $P=G_v^*$, let $R=G_{uv}$ and let $Q=G_{vw}$. Note that $R^\phi=Q$, and $R$ and $Q$ both have index $p$ in $P$. 

Let $N$ be the subgroup of $P$ generated by all the subgroups of $R$ that are normalised by $\phi$. By~\cite[Proposition~2.1]{Glaub}, $N$ is normal in $P$. By definition, $N$ is normalised by $\phi$ and hence $N$ is normalised by $\langle P,\phi \rangle$. On the other hand, we have shown that $\langle P,\phi \rangle$ is transitive on $\V(\Gamma)$.  Since $N\leq P\leq G_v$, it follows that $N=1$. This shows that condition (1.1) of~\cite{Currano} is satisfied. 

Let $u$, $v$ and $x_1,\ldots,x_t$ be as in~\cite[Theorem~1]{Currano} and let $E=\langle x_1,\ldots,x_u\rangle$. By~\cite[Lemma~2.2~(d,f,g)]{Currano}, $P$ has nilpotency class at most $3$, $\Zent{P}=\langle x_{v+1},\ldots,x_u\rangle$ and $\gamma_2(P)$ is elementary abelian. It follows from~\cite[Lemma~2.1]{Currano} that $P$ is generated by elements of order $p$ and that $|P|=p^t$, $|E|=p^u$ and $|\Zent P|=p^{u-v}$. It also follows from~\cite[Theorem~1~(1.2-1.5)]{Currano} that $v=t-u$, $u\geq \frac{2}{3}t$, $E\leq P$ and $E$ is elementary abelian which concludes the proof of (\ref{bonbon}). Since $v=t-u$, $|\Zent P|=p^{u-v}=p^{2u-t}\geq p^{t/3}=|P|^{1/3}$ and (\ref{boubou}) follows. Finally, $P$ has exponent at most $p^2$ by Lemma~\ref{poddtechnical}.

\end{proof}

\section{The graphs $\PX(p,r,s)$}\label{sec:graphsCPRS}
We now define the graphs $\PX(p,r,s)$ which appear in our main theorem and prove some useful results about these graphs. These graphs were first studied by Praeger and Xu~\cite{PraegerXu}. 

Let $p$ be a prime and let $r$ and $s$ be positive integers with $r\geq 3$ and $1\leq s\leq r-1$. The graph $\PX(p,r,1)$ is the \emph{lexicographic product} of a cycle of length $r$ and an edgeless graph on $p$ vertices. In other words, $\V(\PX(p,r,1))=\ZZ_p\times\ZZ_r$ with $(u,i)$ being adjacent to $(v,j)$ if and only if $i-j\in\{-1,1\}$. A path in $\PX(p,r,1)$ is called {\em traversing} if it contains at most one vertex from $\ZZ_p\times\{y\}$ for each $y\in\ZZ_r$. For $s\geq 2$, the graph $\PX(p,r,s)$ has vertex-set the set of traversing paths of $\PX(p,r,1)$ of length $s-1$, with two such paths being adjacent in $\PX(p,r,s)$ if and only if their union is a traversing path of length $s$ in $\PX(p,r,1)$.

It is easy to see that $\PX(p,r,s)$ is a connected $2p$-valent graph with $rp^s$ vertices. There is an obvious action of the wreath product $\Sym(p)\wr \D_r$ as a group of automorphisms of $\PX(p,r,1)$ with an induced faithful arc-transitive action on $\PX(p,r,s)$. Let $H=\C_p\wr \D_r\leq \Sym(p)\wr \D_r$, let $\Gamma=\PX(p,r,s)$ and let $v$ be a vertex of $\Gamma$. It is easily seen that if $s\leq r-2$ then $H_v^{\Gamma(v)}\cong\L_{p,2}$ and $|H_v|= 2p^{r-s}$. Thus, if $p$ and $s$ are fixed then $|H_v|$ grows exponentially with $r$ and hence exponentially  with $|\V(\Gamma)|$. This behaviour contrasts sharply with the upper bounds which are the other possibilities in the conclusion of Theorem~\ref{theorem:main2}. 

One of our main tools is  the following striking result due to Praeger and Xu.

\begin{theorem}[\rm{\cite[Theorem~$1$]{PraegerXu}}]\label{PrePXu}
Let $p$ be a prime and let $\Gamma$ be a connected $G$-arc-transitive graph of valency $2p$. If $G$ has an abelian normal $p$-subgroup which is not semiregular then $\Gamma\cong \PX(p,r,s)$ for some $r\geq 3$ and $1\leq s\leq r-1$. 
\end{theorem}

There is also an analogous result by Praeger about asymmetric arc-transitive digraphs. Note that the cyclic ordering of $\ZZ_r$ induces a natural asymmetric orientation of $\PX(p,r,s)$, which we denote $\vPX(p,r,s)$.

\begin{theorem}[\rm{\cite[Theorem~$2.9$]{PraegerDigraph}}]\label{PraDigraph}
Let $p$ be a prime and let $\vGa$ be an asymmetric connected $G$-arc-transitive digraph of  out-valency $p$. If $G$ has an abelian normal subgroup which is not semiregular then $\vGa\cong \vPX(p,r,s)$ for some $r\geq 3$ and $1\leq s\leq r-1$.
\end{theorem}

These two results can be combined to yield the following.

\begin{corollary}\label{PXu}
Assume Hypothesis~A. If $G$ has an abelian normal subgroup which is not semiregular then $\Gamma\cong \PX(p,r,s)$ for some $r\geq 3$ and $1\leq s\leq r-1$.
\end{corollary}
\begin{proof}
Let $N$ be an abelian normal subgroup of $G$ which is not semiregular. Since $N_v\neq 1$, we have that $N_v^{\Gamma(v)}\neq 1$. In particular, $N_v^{\Gamma(v)}$ is a non-identity abelian normal subgroup of $\L_{p,\chi}$. By Lemma~\ref{lemma:blocks}, it follows that $N_v^{\Gamma(v)}$ is a non-identity $p$-group and hence so is $N_v$. Let $M$ be the group generated by the elements of order $p$ in $N$. This is an elementary abelian $p$-group which is characteristic in $N$ and hence normal in $G$. Moreover, $M_v\neq 1$. If $\chi=2$ then $\Gamma$ is $G$-arc-transitive and Theorem~\ref{PrePXu} concludes the proof. If $\chi=1$ then, by Lemma~\ref{lemma:graphtodigraph}, $\Gamma$ has a $G$-arc-transitive asymmetric orientation $\vGa$ of out-valency $p$. By Theorem~\ref{PraDigraph}, $\vGa\cong \vPX(p,r,s)$ for some $r\geq 3$ and $1\leq s\leq r-1$ and the result follows.
\end{proof}

Let $K_{n,n}$ denote the regular complete bipartite graph on $2n$ vertices.

\begin{theorem}[\rm{\cite[Theorem~$2.13$]{PraegerXu}}]\label{PXuAutomorphism}
Let $\Gamma=\PX(p,r,s)$ with $r\geq 3$ and $s\leq r-1$, let $X=\Sym(p)\wr \D_r$ and let $A=\Aut(\Gamma)$. Then one of the following occurs:
\begin{enumerate}
\item $A=X$;
\item $r=4$, $s=3$, $p=2$ and $|A:X|=2$;
\item $r=4$, $s=2$,  $p=2$ and $|A:X|=3$;
\item $r=4$, $s=1$, $\Gamma\cong K_{2p,2p}$ and $A=\Sym(2p)\wr\C_2$.
\end{enumerate}
\end{theorem}

\begin{corollary}\label{RealCorollary}
Assume Hypothesis~A. If $G$ contains an abelian normal subgroup that is not semiregular then $\Gamma\cong \PX(p,r,s)$ for some $r\geq 3$ and $1\leq s\leq r-2$. 
\end{corollary}
\begin{proof}
By Corollary~\ref{PXu},  $\Gamma\cong \PX(p,r,s)$ for some $r\geq 3$ and $1\leq s\leq r-1$. We thus assume that $s=r-1$ and we will obtain a contradiction. Let $A=\Aut(\Gamma)$, let $Y=\Sym(p)^r$ and let $X=Y\rtimes \D_r$.

We first assume that $X=A$ and hence $A_v=Y_v\rtimes \C_2$. Note that a Sylow $p$-subgroup of $Y$ has order $p^r$ while the orbits of $Y$ have size $p^s$. It follows that a Sylow $p$-subgroup of $Y_v$ has order $p^{r-s}=p$. Since $p^2$ divides $|G_v^*|$ and thus $|A_v|$, it follows that $p=2$, $|A_v|=4$, $G_v^*=A_v$ and $\chi=1$. On the other hand, $(G_v^*)^{\Gamma(v)}$ is intransitive while $A_v^{\Gamma(v)}$ is transitive, a contradiction.

Suppose now that $X<A$. By Theorem~\ref{PXuAutomorphism}, we have $r=4$ and, since $s=r-1$, we have also $s=3$, $p=2$ and $\Gamma\cong\PX(2,4,3)$. We show that every abelian normal subgroup of $G$ is semiregular. By Theorem~\ref{PXuAutomorphism}, $|A_v|=2|X_v|=8$. In particular, if $\chi=2$ then $G_v=A_v$ and (since $G$ is vertex-transitive) $G=A$. A computer-assisted approach (using \texttt{magma}~\cite{magma} for example) can be used to check that every abelian normal subgroup of $A$ is semiregular. If $\chi=1$ then $G$ is not arc-transitive and, since $G$ is vertex-transitive and $|G_v|\geq 4$, it follows that $|A:G|=2$. Again, using a computer-assisted approach reveals that $A$ has a unique subgroup of index $2$ that is vertex- and edge-transitive but not arc-transitive. It is then straightforward to check that every abelian normal subgroup of this group is semiregular.
\end{proof}

\begin{corollary}\label{cor:PXuAutomorphism}
Assume Hypothesis~A. If $\Gamma\cong\PX(p,r,s)$ for some $r\geq 3$ and $1\leq s\leq r-2$ then $G$ is conjugate to a subgroup of $\Sym(p)\wr \D_r$. 
\end{corollary}
\begin{proof}
Let $X=\Sym(p)\wr \D_r$ and let $A=\Aut(\Gamma)$. If $A=X$ then the result is obvious. We may thus assume that $X<A$. By Theorem~\ref{PXuAutomorphism}, it follows that $r=4$ and $s\in\{1,2\}$. If $p=2$ then $|\V(\Gamma)|$ and $|G_v|$ are both powers of $2$ and hence so is $|G|$. Moreover, it follows from Theorem~\ref{PXuAutomorphism} that $X$ is a Sylow $2$-subgroup of $A$ and hence $G$ is conjugate to a subgroup of $X$.

We thus assume that $p$ is odd and hence $s=1$, $\Gamma\cong K_{2p,2p}$ and $A=\Sym(2p)\wr\C_2$. Let $P$ be a Sylow $p$-subgroup of $G$. Since $\Gamma$ is $G$-vertex-transitive, $P$ has four orbits of size $p$.  We claim that $G$ contains a normal subgroup $N$ contained in $P$ and having the same orbits as $P$. Note that the result immediately follows from this claim because $G\leq \norm A N\leq \Sym(p)\wr \D_4$.

If $P$ is normal in $G$ then the claim is clearly true, thus we assume that $P$ is not normal in $G$. Let $G^+$ be the index $2$ subgroup of $G$ preserving the bipartition of $\Gamma$. Note that $P$ is a Sylow $p$-subgroup of $G^+$. If $P$ is normal in $G^+$ then it is in fact characteristic in $G^+$ and thus normal in $G$. We may thus assume that $P$ is not normal in $G^+$.

Note that $G_v\leq G^+$. Since $p$ is odd, $G_v^*$ is the unique Sylow $p$-subgroup of $G_v$ and thus $G_v^*=P_v$. It follows that $|G^+:P_v|=|G^+:G_v||G_v:P_v|=(2p)\chi$ and thus  $|G^+:P|=2\chi$. Since $P$ is not normal in $G^+$, Sylow's theorems imply that $p=3$, $\chi=2$ and $G^+$ has exactly four Sylow $3$-subgroups. Let $N$ be the core of $P$ in $G^+$. Note that, since $G^+$ has four Sylow $3$-subgroups, $|G^+:N|$ divides $4!=24$ and hence, as $P$ is a $3$-group, $|P:N|=3$. Moreover, $N$ is a characteristic subgroup of $G^+$ and thus a normal subgroup of $G$. It remains to show that $N$ has the same orbits as $P$. Since $P$ has four orbits of size $3$, the only other possibility is that $N$ fixes a point. Since it is normal in the vertex-transitive group $G$, this would imply that $N=1$ and $|P|=3$, contradicting the fact that $P$ contains the subgroup $G_v^*$ of order at least $9$.
\end{proof}

\begin{lemma}\label{lemma:againC}
Assume Hypothesis~A. Assume that $\Gamma\cong\PX(p,r,s)$ for some $r\geq 3$ and $1\leq s\leq r-2$. Let $E$ be the normal closure of $G_v^*$ in $G$. Then $E$ is an elementary abelian $p$-group and $\Gamma/E$ is a cycle of length at least $3$ on which $G/E$ acts faithfully. Moreover, $E_v=G_v^*$.
\end{lemma}
\begin{proof}

Let $Y=\Sym(p)^r$ and let $X=Y\rtimes \D_r$. By Corollary~\ref{cor:PXuAutomorphism}, we may assume that $G\leq X$. Since $s\leq r-2$, $Y_v^{\Gamma(v)}$ has two orbits. Note that the two orbits of $Y_v^{\Gamma(v)}$ form a system of imprimitivity for $X_v^{\Gamma(v)}$ and hence for $G_v^{\Gamma(v)}\cong\L_{p,\chi}$. By Lemma~\ref{lemma:blocks}, it follows that the orbits of $Y_v^{\Gamma(v)}$ are exactly the orbits of $(G_v^*)^{\Gamma(v)}$. Since $|X_v:Y_v|=2$, every element of $X_v \setminus Y_v$ interchanges the orbits of $Y_v^{\Gamma(v)}$; hence $G_v^*\leq Y_v$ and $G_v^*\leq (G\cap Y)_v\leq G_v$. As $(G\cap Y)_v^{\Gamma(v)}$ is intransitive, it follows by Lemma~\ref{lemma:Gv*} that $(G\cap Y)_v=G_v^*$. In particular, $(G\cap Y)_v$ is a $p$-group.

As $Y$ is normal in $X$, $G\cap Y$ is normal in $G$ and thus $E\leq G\cap Y$. Note that the orbits of $Y$ have size a power of $p$, hence the orbits of $G\cap Y$ also have size a power of $p$. We have already seen that $(G\cap Y)_v$ is a $p$-group therefore so is $G\cap Y$. Note that a Sylow $p$-subgroup of $Y$ is elementary abelian and hence so are $G\cap Y$ and $E$.

Since $1<G_v^*\leq E_v$, it follows from Lemma~\ref{lemma:quotients} that $\Gamma$ is not a regular cover of $\Gamma/E$. Moreover, $E\leq Y$ and hence $E$ has at least $r\geq 3$ orbits. By Lemma~\ref{lemma:easy}, we conclude that $\Gamma/E$ is a cycle of length at least $3$.

Let $K$ be the kernel of the action of $G$ on $E$-orbits. Note that $K=EK_v$. By Lemma~\ref{lemma:cyclecycle}, $K_v=G_v^*\leq E$; hence $E=K$ and $E_v=G_v^*$ which concludes the proof.
\end{proof}

\section{$\Gamma/N\cong\PX(p,r,s)$}\label{sec:nc}

In this section, we consider a rather specific but important case towards the proof of Theorem~\ref{theorem:main2}. We assume Hypothesis~A and, moreover, we assume that $G$ has a normal subgroup $N$ with $\Gamma/N\cong\PX(p,r,s)$. We show that $(\Gamma,G)$ satisfies either Theorem~\ref{theorem:main2}~(\ref{labelmain1}) or~(\ref{mainbound}). Recall that $\Op p G$ denotes the largest normal $p$-subgroup of $G$.

\begin{theorem}\label{Theorem:PreWreathCover}
Assume Hypothesis~A. Suppose that $\Op p G =1$ and that $G$ has an elementary abelian minimal normal $q$-subgroup $N$. Assume that $\Gamma/N\cong \PX(p,r,s)$ for some $r\geq 3$ and $1\leq s\leq r-2$. Then $|\V(\Gamma)|>2p|G_v^*|\log_p|G_v^*|$.
\end{theorem}
\begin{proof}
It follows from Lemma~\ref{lemma:easy} that $\Gamma$ is a regular cover of $\Gamma/N$. By Corollary~\ref{lemma:cover}, Hypothesis~A holds with $(\Gamma,G)$ replaced by $(\Gamma/N,G/N)$. It then follows from Lemma~\ref{lemma:againC} that $G/N$ contains a normal $p$-subgroup $E/N$ such that  $(\Gamma/N)/(E/N)$ is a cycle of length at least $3$ on which $(G/N)/(E/N)$ acts faithfully. Note that $\Gamma/E\cong(\Gamma/N)/(E/N)$ and $G/E\cong(G/N)/(E/N)$. By Lemma~\ref{lemma:cyclecycle}, we have $G_v^*=E_v$. Let $|E_v|=p^t$. Observe that $q\neq p$ because $\Op p G = 1$.

Since $\Gamma$ is a regular cover of $\Gamma/N$, Lemma~\ref{lemma:quotients} (\ref{a9}) implies that $|(E/N)_{v^N}|=|E_v|=p^t$. As $|\V(\Gamma/N)|>|\V((\Gamma/N)/(E/N))|$, it follows that $|(v^N)^{E/N}|> 1$ and, since $E/N$ is a $p$-group, $|(v^N)^{E/N}|\geq p$. This implies that $|E/N|=|(v^N)^{E/N}||(E/N)_{v^N}|\geq p^{t+1}$.

Let $C$ be the centraliser of $N$ in $E$. Since $N$ and $E$ are normal in $G$, so is $C$. Let $K$ be a Sylow $p$-subgroup of $C$. Since $E/N$ is a $p$-group, we have $C=NK$. As $K$ centralises $N$, it follows that $K$ is normal and thus characteristic in $C$. In particular, $K$ is normal in $G$. Since $\Op p G =1$, we get $K=1$ and $C=N$. In particular, $E/N$ acts faithfully on $N$ by conjugation. As $|E/N|\geq p^{t+1}$, Lemma~\ref{PabloLemma} implies that $|N|\geq (p+1)^{t+1}$. By Lemma~\ref{lemma:stupidtechnical}, we have
$$|\V(\Gamma)|=|\V(\Gamma/N)||N|=rp^s|N|\geq (3p)(p+1)^{t+1}>(3p)(p+t+1)p^t>2tp^{t+1}=2p|G_v^*|\log_p|G_v^*|.$$ 
\end{proof}

If so minded, one could remove the hypothesis that $N$ is abelian from Theorem~\ref{Theorem:PreWreathCover} using Lemma~\ref{neandertal1}.

\begin{theorem}\label{Theorem:WreathCover}
Assume Hypothesis~A. If $G$ has an elementary abelian minimal normal $p$-subgroup $N$ such that $\Gamma/N\cong \PX(p,r,s)$ for some $r\geq 3$ and $1\leq s\leq r-2$ then one of the following holds:
\begin{enumerate}
\item $\Gamma\cong \PX(p,r',s')$ for some $r'\geq 3$ and $1\leq s'\leq r'-2$;\label{jaja}
\item $|\V(\Gamma)|\geq 2p|G_v^*|\log_p|G_v^*|$.\label{jajajaja}
\end{enumerate}

\end{theorem}

\begin{proof}
It follows from Lemma~\ref{lemma:easy} that $\Gamma$ is a regular cover of $\Gamma/N$ and $N$ is semiregular. By Corollary~\ref{lemma:cover}, Hypothesis~A holds with $(\Gamma,G)$ replaced by $(\Gamma/N,G/N)$. It then follows from Lemma~\ref{lemma:againC} that $G/N$ contains an elementary abelian normal $p$-subgroup $E/N$ such that  $E/N$ is equal to the normal closure of $(E/N)_{v^N}=E_vN/N$ in $G/N$ and, moreover,  $(\Gamma/N)/(E/N)$ is a cycle of length $m\geq 3$ on which $(G/N)/(E/N)$ acts faithfully. Note that $\Gamma/E\cong(\Gamma/N)/(E/N)$,  $G/E\cong(G/N)/(E/N)$ and $|\V(\Gamma)|=m|v^E|$. As $N$ is semiregular, $ E_vN/N\cong E_v$ and, since $E/N$ is elementary abelian, so is $E_v$.

Fix an asymmetric orientation of the cycle $\Gamma/E$. This induces an asymmetric orientation of the graph $\Gamma$, yielding an asymmetric digraph $\vGa$ of in- and out-valency $p$, whose underlying graph is $\Gamma$. Let $G^+/E$ be the orientation-preserving subgroup of $G/E$. By Lemma~\ref{lemma:cyclecycle}, we have $G_v^*=G_v^+ =E_v$, $|G:G^+|=\chi$, $\vGa$ is $G^+$-half-arc-transitive, $(G^+_v)^{\vGa^+(v)}\cong\C_p$  and $G^+/E\cong\C_m$. 

Since $G_v^*= E_v$, we have $E_v\neq 1$. Let $F$ be the normal closure of $E_v$ in $G$. As $N$ is a minimal normal subgroup of $G$, we obtain that either $N\cap F=1$ or $N\leq F$. If $N\cap F=1$ then $F\cong FN/N\leq E/N$. It follows that $F$ is a normal elementary abelian $p$-subgroup of $G$. As $E_v\neq 1$, we have $F_v\neq 1$ and hence it follows from Corollary~\ref{RealCorollary} that~(\ref{jaja}) holds. We may therefore assume that $N\leq F$. As the normal closure of $E_vN$ in $G$ is $E$, we have $E=(E_vN)^G=(E_v)^GN=FN=F$, that is, $E=(E_v)^G$.

Let $t$ be the largest integer such that $E_v$ acts transitively on the $t$-arcs of $\vGa$ starting at $v$ and let $(v_0,\ldots,v_t)$ be such a $t$-arc (and thus $v=v_0$). Since $E_v$ is transitive on the $t$-arcs of $\vGa$ starting at $v$ and $G^+$ is vertex-transitive, $G^+$ is transitive on $t$-arcs of $\vGa$.

For $0\leq i\leq t$, let $E_i$ be the pointwise stabiliser of $\{v_0,...,v_{t-i}\}$.  If $(E_0)^{\vGa^+({v_t})}$ is transitive, then $E_v$ is transitive on the $(t+1)$-arcs starting at $v$, contradicting the maximality of $t$. It follows that $(E_0)^{\vGa^+({v_t})}$ is intransitive. On the other hand, $(E_0)^{\vGa^+(v_t)}\leq (G^+_{v_t})^{\vGa^+({v_t})}\cong\C_p$ and hence $(E_0)^{\vGa^+(v_t)}$ is trivial. In other words, the stabiliser of $(v_0,\ldots,v_t)$ fixes all the out-neighbours of $v_t$. Since $\vGa$ is strongly connected, it follows that $E_0=1$ and hence $|E_i|=p^i$ for $0\le i \le t$. In particular,
$|E_t|=|E_{v}|=p^t$.  As $E_v=G^*_v$, we have $t\geq 2$ and
\begin{equation}
2p|G_v^*|\log_p|G_v^*|=2tp^{t+1}.\label{eq11}
\end{equation}

Since $G^+$ is transitive on $t$-arcs of $\vGa$, there exists $a\in G^+$ such that $(v_0,\ldots,v_t)^a=(v_0^a,v_0,\ldots,v_{t-1})$, that is, $v_i = v_0^{a^{-i}}$ for $0\leq i\leq t$. As $a$ acts as a rotation of order $m$ on $\vGa/E$, we get $G^+=E\langle a\rangle$. Let $x$ be a generator of the cyclic group $E_1$. For any integer $i$, let $x_i=x^{a^i}$ and $v_i=v_0^{a^{-i}}$ (note that this definition of $v_i$ is consistent with the definition of $v_i$ that we had for $0\leq i\leq t$).  We now prove three claims from which the result will follow.

\smallskip

\noindent\textsc{Claim~1. }$E_i=\langle x_0,\ldots,x_{i-1}\rangle$ for $1\leq i\leq t$. 

\noindent We argue by induction on $i$. If $i=1$, then by definition, $x=x_0$ and $E_1=\langle x_0\rangle$. Assume $E_i=\langle x_0,\ldots,x_{i-1}\rangle$ for some $i$ with $1\leq i\leq t-1$. As $x$ fixes $\{v_0,\ldots,v_{t-1}\}$ pointwise  and $v_t^{x}\neq v_t$, the element $x_i=x^{a^i}$ fixes $\{v_0^{a^i},\ldots,v_{t-1}^{a^i}\}$ pointwise  and $(v_t^{a^i})^{x_i}\neq v_t^{a^i}$, that is, $x_i$ fixes $\{v_{-i},\ldots,v_{-i+t-1}\}$ pointwise and $v_{-i+t}^{x_i}\neq v_{-i+t}$. In particular, by definition of $E_{i+1}$, we get  $x_i\in E_{i+1}\setminus E_i$. As $|E_{i+1}:E_i|=p$, we obtain $E_{i+1}=E_i\langle x_i\rangle=\langle x_0,\ldots,x_i\rangle$, completing the induction.~$_\blacksquare$

\smallskip

For any positive integer $i\geq 1$, we define $E_i=\langle x_0,\ldots,x_{i-1}\rangle$. (Claim~$1$ shows that, for $1\leq i\leq t$, this definition is consistent with the original definition of $E_i$.) Note that, for any $i\geq 1$, $E_i\leq \langle E_i,E_i^a\rangle=E_{i+1}$. Since $E$ is finite, there exists a smallest $e\geq 0$ such that $E_{t+e}=E_{t+e+1}$. 
\smallskip

\noindent\textsc{Claim~$2$. }$E=E_{t+e}$. 

\noindent Since $E_{t+e}=E_{t+e+1}=\langle E_{t+e},E_{t+e}^a\rangle$, it follows that $E_{t+e}$ is normalised by $a$. Moreover, since $E_v=E_t$ and $\vGa$ is an asymmetric $G^+$-arc-transitive digraph and $a$ maps $v$ to an adjacent vertex, we have that $G^+=\langle G^+_v,a\rangle=\langle E_t,a\rangle$. It follows that $E_{t+e}$ is normalised by $G^+$. Therefore $E_{t+e}\geq (E_v)^{G^+}=\langle E_w\mid w\in \V(\Gamma)\rangle=(E_v)^{G}=E$. As $E$ is normal in $G$, we have $E_{t+e}\leq E$. ~$_\blacksquare$

\smallskip

Recall that $N$ and $E/N$  are both elementary abelian $p$-groups hence $E$ is a $p$-group. From the definition of $e$, we have $|E_{t+i}:E_{t+i-1}|\geq p$ for $1\leq i\leq e$ therefore $|E_{t+e}:E_t|\geq p^e$. In particular, Claim~$2$ gives

\begin{equation}\label{eq:Thm242}
|v^{E}|=|E:E_v|=|E_{t+e}:E_t|\geq p^e.
\end{equation} 

\noindent\textsc{Claim~$3$. }$m\geq t+e$.

\noindent Assume, by contradiction, that $m<t+e$. In particular,  $E=E_{t+e}=\langle x_0,\ldots,x_{t+e-m-1},\ldots,x_{t+e-1}\rangle$. Since $G^+/E$ is a cyclic group of order $m$ and $a\in G^+$, we get $a^m\in E$ but $x_{t+e-1}=x_{t+e-m-1}^{a^m}$ and hence, by Lemma~\ref{Pablolemma2}, we have $E_{t+e}=\langle x_0,\ldots, x_{t+e-2}\rangle=E_{t+e-1}$, contradicting  the minimality of $e$.~$_\blacksquare$

\smallskip

Let  $\Zent E$ be the centre of $E$. If $\Zent E$ is not semiregular then~(\ref{jaja}) follows from Corollary~\ref{RealCorollary}. Therefore we may assume that $\Zent E$ is semiregular. Recall that $E_v=E_t=\langle x_0,\ldots,x_{t-1}\rangle$ is abelian and hence $E_t^{a^{t-1}}=\langle x_{t-1},\ldots,x_{2t-2}\rangle$ is also abelian. Therefore $x_{t-1}$ is central in $\langle E_t,E_t^{a^{t-1}}\rangle=\langle x_0,\ldots,x_{2t-2}\rangle=E_{2t-1}$. Since $x_{t-1}\in E_v$ and $\Zent E\cap E_v=1$, we get $E_{2t-1}<E=E_{t+e}$ and hence $2t-1<t+e$ from which it follows that $e\geq t$. 

Assume $e\geq t+1$. From~$(\ref{eq:Thm242})$ and Claim~$3$, we have $|\V(\Gamma)|=m|v^E|\geq (t+e)p^e\geq (2t+1)p^{t+1}$ and~(\ref{jajajaja}) follows from~\eqref{eq11}. Therefore, from now on, we may assume that $e=t$ and, in particular,
\begin{equation}\label{eq:6}
E=E_{2t}=\langle x_0,\ldots,x_{t-1},x_t,\ldots,x_{2t-1}\rangle=\langle E_v,E_v^{a^t}\rangle.
\end{equation}

Let $X=\Zent E E_v^{a^t}\cap E_v$. Since $E_v$ and $E_v^{a^t}$ are abelian subgroups of $E$, we have $[X,E_v^{a^t}]\leq [\Zent E E_v^{a^t},E_v^{a^t}]=1$ and $[X,E_v]\leq [E_v,E_v]=1$. Hence, by~$(\ref{eq:6})$, we obtain $[X,E]=1$ and thus $X\leq \Zent E \cap E_v=1$. It follows that 
\begin{eqnarray*}
|\Zent E E_v^{a^t}E_v|&=&\frac{|\Zent E E_v^{a^t}||E_v|}{|X|}=|\Zent E E_v^{a^t}||E_v|\\
&=&\frac{|\Zent E ||E_v^{a^t}|}{|\Zent E \cap E_v^{a^t}|}p^t=|\Zent E ||E_v^{a^t}|p^t=|\Zent E|p^{2t}.\nonumber
\end{eqnarray*} 

In particular, $|E|\geq |\Zent E|p^{2t}\geq p^{2t+1}$ and hence 
$$|\V(\Gamma)|=m|v^E|=m|E:E_v|\geq 2t p^{t+1}=2p|G_v^*|\log_p|G_v^*|.$$
\end{proof}

\begin{remark}\label{remark:equality}
By going through the proofs of our main theorems in Section~\ref{sec:mainproofs}, one can check that the only way the inequality in Theorem~\ref{theorem:main2}~(\ref{mainbound}) can be met is if the inequality in  Theorem~\ref{Theorem:WreathCover}~(\ref{jajajaja}) is met. It is thus important to note that the proof of Theorem~\ref{Theorem:WreathCover} in fact gives a great deal of information about $(\Gamma,G)$ in this situation.

For example, $G$ contains a normal $p$-subgroup $E$ with $\Gamma/E$ a cycle of length $2t$ and $G/E$ isomorphic to either $\C_{2t}$ or $\D_{2t}$ for some $t\geq 2$, depending on whether $\chi=1$ or $\chi=2$. Moreover, $E/\Zent E$ is an elementary abelian $p$-group and $|\Zent E|=p$, and hence $E$ is an extraspecial $p$-group. 

A complete classification of all possible pairs $(\Gamma,G)$ meeting this bound in the case $(p,\chi)=(2,2)$ was obtained in~\cite{PSV4valent} and the pairs which arise are described in~\cite{PSVTetravalentExamples}. With some work, it is likely that a similar classification could be obtained for general $(p,\chi)$ along the same lines. However, we do not take this detour here. 
\end{remark}

\section{Proofs of the main theorems}\label{sec:mainproofs}

In this section, we prove Theorems~\ref{theorem:main2},~\ref{theorem:main1},~\ref{thm:mainbasic},~\ref{GarPraCycle} and Proposition~\ref{prop:explicit}. In light of Lemma~\ref{lemma:Gv*}, we write $|G_v^*|$ rather than $|G_v|/\chi$ in this section.

\begin{proof}[Proof of Theorem~$\ref{theorem:main1}$]
The proof goes by induction on $|\V(\Gamma)|$. 

If $G$ has a non-abelian minimal normal subgroup then, by Corollary~\ref{corollary:nonabelian},  $G$ contains a normal subgroup $N$ such that $\Gamma$ is a regular cover of $\Gamma/N$ and $G/N$ has a unique minimal normal subgroup and this subgroup is non-abelian. If $N=1$ then $(\ref{semisimple})$ holds, otherwise~(\ref{funny}) holds.

We may thus assume that every minimal normal subgroup of $G$ is abelian. Let $N$ be a minimal normal subgroup of $G$. Note that $N$ is elementary abelian. By choosing $N$ appropriately, we can ensure that either $N$ is a $p$-group or $\Op p G =1$.

We may also assume that (\ref{PXU}) does not hold. In particular, by Corollary~\ref{RealCorollary}, $N$ is semiregular. By Lemma~\ref{lemma:easy}~(\ref{reduction}) either (\ref{twoorbits}) or (\ref{cycle}) hold or $\Gamma$ is a regular cover of $\Gamma/N$. We therefore assume the latter.

By Corollary~\ref{lemma:cover}, Hypothesis~A holds with $(\Gamma,G)$ replaced by $(\Gamma/N,G/N)$. By induction, the conclusion of Theorem~\ref{theorem:main1} holds with $(\Gamma,G)$ replaced by $(\Gamma/N,G/N)$.

If $(\Gamma/N,G/N)$ satisfies (\ref{PXU}), that is $\Gamma/N\cong \PX(p,r,s)$ for some $r\geq 3$ and $1\leq s\leq r-2$ then, by Theorems~\ref{Theorem:PreWreathCover} and~\ref{Theorem:WreathCover}, one of (\ref{PXU}) or (\ref{bound}) holds. If $(\Gamma/N,G/N)$ satisfies (\ref{bound}) then (\ref{bound}) holds by Lemma~\ref{lemma:quotients}~(\ref{a9}).  If $(\Gamma/N,G/N)$ satisfies (\ref{twoorbits}), (\ref{cycle}) or (\ref{semisimple}) then, since $N$ is abelian, we have that (\ref{funny}) holds. Finally, note that if $G/N$ has a normal subgroup $M/N$ such that $\Gamma/N$ is a regular cover of $(\Gamma/N)/(M/N)$ then in fact $\Gamma$ is a regular cover of $\Gamma/M$. It follows that if $(\Gamma/N,G/N)$ satisfies~(\ref{funny}) then so does $(\Gamma,G)$. 
\end{proof}

\begin{proof}[Proof of Theorem~$\ref{thm:mainbasic}$]
Let $N$ be a semiregular abelian normal subgroup of $G$ having at most two orbits and let $G_v^{[1]}$ be the kernel of the action of $G_v$ on $\Gamma(v)$. We show that $G_v^{[1]}=1$, that is, $G_v$ acts faithfully on $\Gamma(v)$. Note that this implies that $G_v\cong \L_{p,\chi}$ and $|G_v|=\chi p^2$.

If $N$ is transitive then, since it is semiregular, it is regular. It follows that $N$ has $2p$ orbits on the arcs of $\Gamma$ and that these $N$-orbits are $(v,u)^N$, as $u$ runs through the elements in $\Gamma(v)$. Observe that $G_v^{[1]}$ fixes $\Gamma(v)$ pointwise  and hence $G_v^{[1]}$ fixes each $N$-orbit on arcs. As $N\unlhd G$ and $G$ is vertex-transitive, $G_u^{[1]}$ fixes each $N$-orbit on arcs, for every $u\in \V(\Gamma)$. Let $u$ be a neighbour of $v$. Clearly, $G_v^{[1]}\leq G_u$. Moreover, since $G_v^{[1]}$ fixes each of the $2p$ $N$-orbits on arcs, it follows that $G_v^{[1]}\leq G_u^{[1]}$. Using the connectedness of $\Gamma$, by repeating this argument we get $G_v^{[1]}=1$.

If $N$ has two orbits then, since $N$ is normal in $G$ and since $G$ is edge-transitive, it follows that the orbits of $N$ form a bipartition of $\Gamma$. Since $N$ is semiregular, $N$ has precisely $2p$ orbits on the edges of $\Gamma$. The same argument as in the previous paragraph (with the role of arcs being played by edges) again yields  that $G_v^{[1]}=1$.
\end{proof}

\begin{proof}[Proof of Theorem~$\ref{GarPraCycle}$]

Clearly, $K=N\rtimes K_v$. Moreover, $\Gamma/K=\Gamma/N$ is a cycle of length $m\geq 3$ on which  $G/K$ acts faithfully. It follows from Lemma~\ref{lemma:cyclecycle} that $G_v^*=K_v$ and $K_v$ is a $p$-group. By Lemma~\ref{lemma:Malnic}, $K_v$ is elementary abelian of order at most $p^m$. This implies that $t\leq m$.

Let $C$ be the centraliser of $N$ in $K$. As $N$ is abelian, we have $N\leq C$. Also, as $N\unlhd G$ and $K\unlhd G$, we get $C\unlhd G$. Since $N$ is abelian and $K$ preserves the $N$-orbits setwise, we must have $C^\Delta=N^\Delta$ for each $N$-orbit $\Delta$. (As usual, $C^\Delta$ and $N^\Delta$ denote the permutation groups induced by $C$ and $N$ on $\Delta$.) It follows that $[C,C]$ fixes each $N$-orbits pointwise and hence $[C,C]=1$. Thus $C$ is abelian.

If $C$ is not semiregular then~(\ref{first}) follows from Corollary~\ref{RealCorollary}. We may thus assume that $C$ is semiregular. As $K=NK_v$ and $N\leq C\leq K$, this implies that $C=N$ and hence $K_v$ acts by conjugation faithfully on $N$. Observe that $\Zent K=\cent K K\leq \cent K N=N$. Moreover, since $K_v\neq 1$, we have $\Zent K< N$. Since $N$ is minimal normal in $G$, we obtain $\Zent K=1$ and thus $p\neq q$. By Lemma~\ref{PabloLemma}, we have $|N|\geq q^{t\ord_p(q)}$ and hence $|\V(\Gamma)|=m|N|\geq mq^{t\ord_p(q)}$.
\end{proof}

We are now ready to prove Theorem~\ref{theorem:main2} and Proposition~\ref{prop:explicit}. Note that this proof depends on Theorem~\ref{theo:semisimple} and Corollary~\ref{cor : (2,1)} which will be proved in Sections~\ref{sec:semisimple} and~\ref{sec:(2,1)} but whose proofs do not depend on Theorem~\ref{theorem:main2}. (We use this nonlinear order to improve the flow of the paper, as the proof of Theorem~\ref{theo:semisimple} is quite long and somewhat different in character from the rest of the paper.)

\begin{proof}[Proof of Theorem~$\ref{theorem:main2}$ and Proposition~$\ref{prop:explicit}$]
Let $e(p)$ and $c(p)$ be as in Proposition~\ref{prop:explicit}. If $(p,\chi)=(2,1)$ then the conclusion follows from Corollary~\ref{cor : (2,1)}. If $(p,\chi)=(2,2)$ then the conclusion follows from~\cite[Theorem~1.2]{PSV4valent}. We may thus assume that $p$ is odd. Let $|G_v^*|=p^t$. We will prove the following claim:

\smallskip
\noindent\textsc{Claim. } If neither~(\ref{labelmain1}) or~(\ref{mainbound}) holds then $t< e(p)$. 
\smallskip

\noindent The proof goes by induction on $|\V(\Gamma)|$. We apply Theorem~\ref{theorem:main1}. Since neither~(\ref{labelmain1}) or~(\ref{mainbound}) holds, neither of Theorem~\ref{theorem:main1}~(\ref{PXU}) or~(\ref{bound})  holds. If Theorem~\ref{theorem:main1}~(\ref{twoorbits}) holds then the claim follows from Theorem~\ref{thm:mainbasic} and noting that $2< e(p)$. 

If Theorem~\ref{theorem:main1}~(\ref{cycle}) holds then it follows from Theorem~\ref{GarPraCycle} that $|\V(\Gamma)|\geq tq^{t\ord_p(q)}$ for some prime $q$ different than $p$. Since $q^{t\ord_p(q)}\geq p+1$, we have $|\V(\Gamma)|\geq t(p+1)^t$. Since~(\ref{bound}) does not hold we have  $2tp^{t+1}=2p|G_v^*|\log_p|G_v^*|>|\V(\Gamma)|\geq t(p+1)^t$ and hence $2p^{t+1}>(p+1)^t$. It follows that $t<\frac{\log(2p)}{\log(1+1/p)}=e(p)$.

If Theorem~\ref{theorem:main1}~(\ref{semisimple}) holds then the claim follows from Theorem~\ref{theo:semisimple}. We may thus assume that Theorem~\ref{theorem:main1}~(\ref{funny}) holds. In particular, $G$ has a non-identity normal subgroup $N$ such that $\Gamma$ is a regular cover of $\Gamma/N$ and one of Theorem~\ref{theorem:main1}~(\ref{twoorbits}),~(\ref{cycle}), or (\ref{semisimple}) is satisfied with $(\Gamma,G)$ replaced by $(\Gamma/N,G/N)$. The claim then follows by induction.~$_\blacksquare$

\smallskip

If neither~(\ref{labelmain1}) or~(\ref{mainbound}) holds then our claim implies $|\V(\Gamma)|<2tp^{t+1}< 2e(p)p^{e(p)+1}=c(p)$.
\end{proof}

\begin{remark} The upper bound in Theorem~\ref{theorem:main2}~(\ref{small}) given by Proposition~\ref{prop:explicit} was chosen for simplicity, not strength. Given a fixed prime $p$, it is often possible to greatly improve upon it by using Theorems~\ref{theorem:main1},~\ref{thm:mainbasic},~\ref{GarPraCycle} and~\ref{theo:semisimple} directly.

For example, suppose that $p=19$. Following the proof of Theorem~\ref{theorem:main2}, we see that a basic exceptional pair $(\Gamma,G)$ must satisfy one of Theorem~\ref{theorem:main1}~(\ref{twoorbits}),~(\ref{cycle}) or~(\ref{semisimple}). In the first case, we have $t=2$. In the second case, we have $2t\cdot19^{t+1}>|\V(\Gamma)|\geq tq^{t\ord_{19}(q)}$ for some prime $q\neq 19$. The smallest prime power which is congruent to $1\pmod {19}$ is $191$ hence $q^{\ord_{19}(q)}\geq 191$, from which it follows  that $2\cdot 19^{t+1}> 191^t$ which is a contradiction (since $t\geq 2$). Finally, the third case does not occur by Theorem~\ref{theo:semisimple}. We thus find that $t=2$, $|G_v^*|=19^2$ and $|\V(\Gamma)|<27436$.
\end{remark}

\section{Proof of Theorem~\ref{theo:semisimple}}\label{sec:semisimple}
As Theorem~\ref{lemma:ea}~\eqref{bonbon} indicates, $G_v^*$ must contain a rather large elementary abelian $p$-subgroup. This motivates the introduction of the following definition.

\begin{definition}\label{def:eg}
Given a prime $p$ and a group $G$, the \emph{$p$-rank} $r_p(G)$ of $G$ is the minimal number of generators of an elementary abelian $p$-subgroup of $G$ of maximum order.
\end{definition}

We will need the following lemma describing the $p$-rank of a wreath product.

\begin{lemma}\label{neandertal1}
Let $p$ be a prime,  let $H$ be a group, let $K$ be permutation group on $\Delta$ and let $W=H\wr_{\Delta} K$. Then 
\[
r_p(W)=
\begin{cases}
r_p(K)&\mathrm{if }\,\,p\,\,\,\mathrm{does\, not\, divide}\,\,|H|,\\
|\Delta|r_p(H)&\mathrm{if }\,\,p\,\,\mathrm{divides }\,\,|H|.
\end{cases}
\]
\end{lemma}
\begin{proof}
Let $\Omega=H\times \Delta$ and consider $W$ as a permutation group on $\Omega$, where $H$ acts on the set $H$ by right multiplication. Let $\Sigma=\{H\times \{\delta\}\mid \delta\in \Delta\}$ and note that this is a system of imprimitivity for $W$. We identify $\Sigma$ with $\Delta$ by identifying the block $H\times\{\delta\}\in\Sigma$ with $\delta\in \Delta$. In particular, we say that a subgroup of $W$ is transitive on $\Delta$ if it is transitive on $\Sigma$. Let $B$ be the kernel of the action of $W$ on $\Delta$ and note that $B=H^\Delta$ and $W=B\rtimes K$.

If $p$ does not divide $|H|$ then $|B|$ is coprime to $p$ and hence $K$ contains a Sylow $p$-subgroup of $W$. If follows that $r_p(W)=r_p(K)$.

We thus assume that $p$ divides $|H|$. Clearly, $r_p(W)\geq r_p(B)=|\Delta|r_p(H)$. Let $E$ be an elementary abelian $p$-subgroup of $W$ and let $\mathcal{O}_1,\ldots,\mathcal{O}_r$ be the orbits of $E$ on $\Delta$.  We show that $r_p(E)\leq r_p(H)|\Delta|$. Since we are only interested in $r_p(E)$, we replace $K$ by the projection of $E$ on $K$ with no loss of generality.

 Assume first that $r=1$, that is, $E$ acts transitively on $\Delta$. Fix $\delta_0\in \Delta$ and, for every $\delta\in \Delta$, let $e_\delta$ be an element of $E$ such that $\delta^{e_\delta}=\delta_0$. Write $e_\delta=\sigma_\delta g_\delta\in E$ with $g_\delta\in B$ and $\sigma_\delta\in K$. Note that $\delta^{\sigma_\delta}=\delta_0$. Let $\pi_0:E\cap B \to H$ be the projection $f\mapsto f(\delta_0)$ on the $\delta_0$-coordinate of $B$. Let $f$ be an element of $E\cap B$. Since $E$ is abelian, we have $f=f^{e_\delta}=f^{\sigma_\delta g_\delta}=g_\delta^{-1}f^{\sigma_\delta }g_\delta$ and  \[f(\delta)=f(\delta_0^{\sigma_\delta^{-1}})=f^{\sigma_\delta}(\delta_0)=(g_\delta fg_\delta^{-1})(\delta_0)=g_\delta(\delta_0)f(\delta_0)g_\delta^{-1}(\delta_0)=f(\delta_0)^{g_\delta(\delta_0)}.\]
 Thus, for each $\delta\in \Delta$, $f(\delta)$ is conjugate (via $g_\delta^{-1}(\delta_0)$) to $f(\delta_0)$, and hence $f$ is uniquely determined by its value on $\delta_0$. Since the family $(g_\delta)_{\delta\in \Delta}$ does not depend on $f$, we obtain that $E\cap B\cong \pi_0(E\cap B)$ and hence  $E\cap B$ is isomorphic to an elementary abelian subgroup of $H$. Thus $|E\cap B|\leq p^{r_p(H)}$. As $E$ is abelian and transitive on $\Delta$, the group $E/(E\cap B)$ acts regularly on $\Delta$ and hence $|E:E\cap B|=|\Delta|$. It follows that $|E|=|E:E\cap B||E\cap B|\leq |\Delta|p^{r_p(H)}$. Since $|H|$ is divisible by $p$, we have $p\leq p^{r_p(H)}$ and a moment's thought gives $p^{r_p(E)}=|E|\leq |\Delta|p^{r_p(H)}\leq p^{r_p(H)|\Delta|}$, concluding the case $r=1$.

Assume now that $r> 1$. For $i\in \{1,\ldots,r\}$, denote by $E_i$ the projection of $E$ on $H\wr_{\mathcal{O}_i}K$. We have $E\leq E_1\times \cdots \times E_r$ and, using the case $r=1$, we obtain
$$|E|\leq \prod_{i=1}^r|E_i|=\prod_{i=1}^r p^{r_p(E_i)}\leq \prod_{i=1}^rp^{r_p(H)|\mathcal{O}_i|}=p^{r_p(H)\sum_i|\mathcal{O}_i|}=p^{r_p(H)|\Delta|},$$ completing the case $r>1$.
\end{proof}

\begin{lemma}\label{neandertal2}The $p$-rank of $\mathrm{Aut}(\mathrm{Alt}(n))$ is $\lfloor n/p\rfloor$.
\end{lemma}
\begin{proof}
Suppose that $n\neq 6$. Then $\Aut(\mathrm{Alt}(n))=\Sym(n)$. Write $n=qp+r$, with $0\leq r\leq p-1$. A Sylow $p$-subgroup of $\Sym(n)$ fixes $r$ points of $\{1,\ldots,n\}$ and hence is conjugate to a Sylow $p$-subgroup of $\Sym(pq)$. In particular, we may assume that $r=0$. Clearly, $r_p(\Sym(n))\geq n/p$. 

Let $E$ be an elementary abelian $p$-subgroup of $\mathrm{Sym}(n)$ of maximal cardinality. Let $\mathcal{O}_1,\ldots,\mathcal{O}_\ell$ be the orbits of $E$ on $\{1,\ldots,n\}$ and let $E_i$ be the permutation group induced by $E$ on $\mathcal{O}_i$. Since $E_i$ is abelian, it acts regularly on $\mathcal{O}_i$ and hence $|E_i|=|\mathcal{O}_i|$. Moreover, by maximality, we have $E=E_1\times E_2\times \cdots \times E_\ell$.

Note that, if $a$ is a power of $p$ then $a\leq p^{a/p}$. Thus
$$|E|=\prod_{i=1}^\ell|E_i|=\prod_{i=1}^\ell|\mathcal{O}_i|\leq \prod_{i=1}^\ell p^{|\mathcal{O}_i|/p}=p^{\sum_i |\mathcal{O}_i|/p}=p^{n/p}$$
and $r_p(\Sym(n))=r_p(E)\leq n/p$.

When $n=6$, we have $\Aut(\mathrm{Alt}(n))=\mathrm{P}\Gamma\mathrm{L}_2(9)$ and the result follows with a computation.
\end{proof}

\begin{table}[!h]
\begin{center}
\begin{tabular}{|c|c|c|}\hline
$T$& $p$ &$\ell$\\\hline
$\Alt(5)$ &$3$&$1,2,3$\\
$\Alt(5)$ &$5$&$\ell$\\
$\Alt(6)$ &$3$&$\ell$\\
$\Alt(6)$ &$5$&$1,2,3$\\
$\Alt(7)$ &$3$&$1,2$\\
$\Alt(7)$ &$5,7$&$1$\\
$\Alt(8),\Alt(9)$ &$3$&$1$\\\hline
$M_{11}$&$3$&$1$\\
$M_{11}$&$11$&$1,2$\\
$J_1$&$19$&$1$\\
$J_2$&$5$&$1$\\\hline
$\PSL_2(p^f)$&$p$ &$\ell$\\
$\PSL_3(p)$& $3,5, 7, 11, 13, 19$   & $1$ \\
$\PSL_4(p)$&$3, 5, 7$    & $1$ \\
$\PSU_3(p)$&$3,5,7,11,17,23,29$    & $1$ \\
$\PSU_4(3)$&$3$ & $1,2$ \\
$\PSU_4(p)$&$5,7$ & $1$ \\
$\PSp_4(p^f)$&$p$ &$\ell$\\  
$\PSp_6(p)$&$3,5,7$&$1$\\
$G_2(3)$&$3$&$1,2$\\
$G_2(9)$&$3$&$1$\\\hline

$\PSL_2(r^f),\PSL_3(r^f),\PSU_3(r^f)$&$p\neq r$&$\ell$\\
$\PSU_5(2)$&$3$&$1$\\
${}^2B_2(2^3)$&$13$&$1$\\\hline
\end{tabular}
\end{center}
\caption{Exceptional cases in Theorem~\ref{thrm:ss1}}\label{table:bt}
\end{table}

\begin{theorem}\label{thrm:ss1}
Let $p$ be an odd prime, let $\ell\geq 1$, let $T$ be an non-abelian simple group and let $W=\Aut(T)\wr\Sym(\ell)$. Then either $\ell|T|^\ell>6 r_p(W)p^{3r_p(W)+1}$ or $(T,p,\ell)$ appears in Table~$\ref{table:bt}$.
\end{theorem}
\begin{proof}
We first consider the case that $p$ does not divide $|\Aut(T)|$. From Lemmas~\ref{neandertal1} and~\ref{neandertal2}, we have $r_p(W)=r_p(\Sym(\ell))\leq r_p(\Aut(\Alt(\ell)))=\lfloor \ell/p\rfloor$ and therefore $$6r_p(W)p^{3r_p(W)+1}\leq 6\frac{\ell}{p}p^{3\ell/p+1}=6\ell\left( p^{3/p}\right)^\ell\leq 6\ell 3^\ell,$$
where in the last inequality we used the fact that $p^{3/p}\leq 3$ (which is valid since $p\geq 3$). Observe also that, as $|T|\geq 60$, we have  $\ell|T|^\ell> 6\ell 3^\ell$. Therefore the inequality $\ell|T|^\ell>6 r_p(W)p^{3r_p(W)+1}$ is always satisfied. 

In the rest of this proof we assume that $p$ divides $|\Aut(T)|$. From Lemma~\ref{neandertal1}, we have $r_p(W)=\ell r_p(\Aut(T))$ and hence the inequality in the statement of Theorem~\ref{thrm:ss1} becomes:
\begin{equation}\label{eq:eq1}
|T|^\ell> 6 r_p(\Aut(T))p^{3\ell r_p(\Aut(T))+1}.
\end{equation}
We will subdivide the proof depending on the isomorphism class of $T$ using the Classification of the Finite Simple Groups.

\medskip
\textbf{Alternating groups.}
Assume that $T=\Alt(m)$ for some $m\geq 5$. From Lemma~\ref{neandertal2}, we have $r_p(\Aut(T))=\lfloor m/p\rfloor$ and
$$6 r_p(\Aut(T))p^{3\ell r_p(\Aut(T))+1}\leq 6 m p^{3\ell m/p}\leq 6 m3^{\ell m},$$
where again we used the inequality $p^{3/p}\leq 3$. Now it is a computation to verify that $|T|^\ell=(m!/2)^\ell>6 m 3^{\ell m}$ for every $m\geq 11$. Finally, a case by case analysis (using Lemma~\ref{neandertal1}) shows that for $m\leq 10$ either \eqref{eq:eq1} holds or $(T,p,\ell)$ appears in Table~\ref{table:bt}.

\medskip
\textbf{Sporadic groups.}
Assume that $T$ is isomorphic to one of the twenty-six sporadic simple groups. From~\cite[page~viii]{ATLAS}, we see that $|\Aut(T):T|\leq 2$ and, since $p\geq 3$, we have $r_p(\Aut(T))=r_p(T)$. For every prime $p$ and sporadic group $T$, the value of $r_p(T)$ is tabulated in~\cite[Table $5.6.1$]{GLS}. By considering each sporadic group one-by-one, it is easy to check that  either~\eqref{eq:eq1} holds or $(T,p,\ell)$ appears in Table~\ref{table:bt}.

\medskip
\textbf{Groups of Lie type.}
For the rest of this proof we assume that $T$ is a simple group of Lie type defined over the finite field of order $q$. For twisted groups our notation for $q$ is such that ${}^2B_2(q)$, ${}^2G_2(q)$, ${}^2F_4(q)$,
${}^3D_4(q)$, ${}^2E_6(q)$, $\PSU_n(q)$ and $\POmega_{2m}^-(q)$ are the twisted groups contained in $B_2(q)$,
$G_2(q)$, $F_4(q)$, $D_4(q^3)$, $E_6(q^2)$, $\PSL_n(q^2)$ and $\POmega_{2m}^+(q^2)$, respectively. Moreover, taking in account the exceptional isomorphisms between non-abelian simple groups, we will assume that $T$ is not one of $\PSL_2(4)$, $\PSL_2(5)$, $\PSL_2(9)$, $\PSL_4(2)$ or $(\PSp_4(2))'$. 

Write $q=r^f$, with $r$ prime and $f\geq 1$. Define \[
\delta_{f,p}=
\begin{cases} 
1&\mathrm{if}\,\, p~\mathrm{ divides }~f,\\
0&\mathrm{if}\,\, p~\mathrm{does~not~divide }~f,
\end{cases}
\]
and
\[
\varepsilon=
\begin{cases}
1&\mathrm{if}\,\,p=3 \textrm{ and }T=\POmega_8^+(q),\\
0&\mathrm{otherwise}.
\end{cases}
\]

As usual, we follow closely the notation in~\cite{GLS}. Let $T_0$ be the subgroup of $\Aut(T)$ consisting of the inner-diagonal automorphisms of $T$. (For example, if $T=\mathrm{PSL}_n(q)$ then $T_0=\mathrm{PGL}_n(q)$ and if $T=\mathrm{PSU}_n(q)$ then $T_0=\mathrm{PGU}_n(q)$.)
Using the terminology in \cite{GLS}, $\Out(T)$ is the product (in this order) of $T_0$, the group of field automorphisms (which is cyclic) and of the group of graph automorphisms (which has order at most $2$, unless $T=\POmega_8^+(q)$ in which case it has order $6$). This shows that 

\begin{equation}\label{NewEq}
r_p(\Aut(T))\leq r_p(T_0)+\delta_{f,p}+\varepsilon.
\end{equation}

We first consider the case that $r=p$. Then $|T_0:T|$ is coprime to $p$ and hence $r_p(T_0)=r_p(T)$. Observe further that the maximal order of an abelian unipotent subgroup of $T$ is an upper bound for $p^{r_p(T)}$. The maximal order of a unipotent subgroup of non-abelian simple groups of Lie type is known. The resulting upper bounds for $r_p(T)$ as well as the corresponding reference can be found in Table~\ref{lie} (except for the Ree group $T={}^2G_2(3^f)$ where we give the exact value of $p^{r_p(T)}$). 

Using~(\ref{NewEq}), these upper bounds for $r_p(T)$ yield upper bounds for $r_p(\Aut(T))$. It is then a matter of checking each family one-by-one. (In some cases, the following improvements to using~(\ref{NewEq}) are useful: $r_3(\Aut(\PSL_3(3^3)))=6$, $r_3(\Aut(\PSU_3(3^3)))=6$,  $r_3(\Aut(\PSU_4(3^3)))=12$ and $r_3(G_2(3^3))=12$. These can be easily verified using \texttt{magma}~\cite{magma}, for example).

\begin{table}[!h]
\begin{center}
\begin{tabular}{|c|c|c|c|c}\hline
Lie type&Upper bound for $r_p(T)$ &Reference\\\hline
$\PSL_{2m}$&${m^2}f$&\cite[Theorem~$2.1$]{B}\\
$\PSL_{2m+1}$&${m(m+1)}f$&\cite[Theorem~$2.1$]{B}\\
$\POmega_7$&$5f$&\cite[Theorem~$5.1$]{B}\\
$\POmega_{2m+1}~(m\geq 4)$&$({m(m-1)/2+1})f$&\cite[Theorems~$5.2$ and~$5.3$]{B}\\
$\PSp_{2m}~(m\geq 2)$&${m(m+1)}f/2$&\cite[Theorem~$2.5$]{B}\\
$\POmega_{2m}^+~(m\geq 4)$&${m(m-1)}f/2$&\cite[Theorems~$3.1$ and~$3.2$]{B}\\
$E_{6}$&${16}f$&\cite[Table~$4$]{V}\\
$E_{7}$&${27}f$&\cite[Table~$4$]{V}\\
$E_{8}$&${36}f$&\cite[Table~$4$]{V}\\
$F_4$&${9}f$&\cite[Table~$4$]{V}\\
$G_{2}$&$3f$ if $p\neq 3$, $4f$ if $p=3$&\cite[Table~$4$]{V}\\
$\PSU_{2m+1}$&$(m^2+1)f$&\cite[Theorem~$1$]{W}\\
$\PSU_{2m}~(m\geq 2)$&$m^2f$&\cite[Theorem~$1$]{W}\\
$^2B_{2}$&$f$&\cite{S}\\
$\POmega_8^-$&$6f$&\cite[table on page~$230$]{W2}\\
$\POmega_{2m}^-~(m\geq 5)$&$((m-1)(m-2)/2+2)f$&\cite[Theorem~I~(d)]{W2}\\
$^3D_{4}$&$5f$&\cite[Table~$4$]{V}\\
$^2E_{6}$&$12f$&\cite[Table~$4$]{V}\\
$^2F_{4}$&$5f$&\cite[Table~$4$]{V}\\
$^2G_{2}$&$f$&\cite{R}\\\hline
\end{tabular}
\end{center}
\caption{Upper bound for $r_p(T)$ for groups of Lie type with $p=r$ in Theorem~
\ref{thrm:ss1}}\label{lie}
\end{table}

We now consider the case that $r\neq p$. Let $E$ be an elementary abelian $p$-subgroup of maximal order of $T_0$. From~\cite[Theorem~$4.10.2$]{GLS}, it follows that $E$ is contained in a maximal torus $V$ of $T_0$ and that 
\begin{equation}\label{EQ1}
r_p(E)\leq k,
\end{equation} where $k$ is the Lie rank of the algebraic group associated with $T_0$. It is well-known that $|V|\leq (q_0+1)^k$, where $q_0=\sqrt{q}$  when $T$ is  isomorphic to the Suzuki group ${}^2B_2(q)$ or to a Ree group ${}^2G_2(q)$, ${}^2F_4(q)$ and $q_0=q$  otherwise (for a detailed proof tailored to our needs see for example~\cite[Lemma~$2.4$]{SZ}). In particular, 
\begin{equation}\label{EQ2}
|E|\leq (q_0+1)^k
\end{equation} 
and hence
\begin{equation}\label{EQ3}
p\leq (q_0+1)^k.
\end{equation}

When $\delta_{f,p}=\varepsilon=0$, we have $r_p(\Aut(T))=r_p(T_0)=r_p(E)$ and hence, from~\eqref{EQ1},~\eqref{EQ2} and~\eqref{EQ3}, we get
\begin{align*}
6r_p(\Aut(T))p^{3\ell r_p(\Aut(T))+1}&\leq 6k (p^{r_p(E)})^{3\ell}\cdot p=6k |E|^{3\ell}\cdot p\\
&\leq 6k(q_0+1)^{3\ell k}(q_0+1)^k=6k(q_0+1)^{3\ell k+k}.
\end{align*}
When $\delta_{f,p}=1$, we have $p\leq f$ and, by~\eqref{NewEq} and~\eqref{EQ1}, we have  
$$6r_p(\Aut(T))p^{3\ell r_p(\Aut(T))+1}\leq 6 (k+1+\varepsilon) f^{3\ell(k+1+\varepsilon)+1}.$$ 
Finally, when $\delta_{f,p}=0$ and $\varepsilon=1$, we have $T=\POmega_8^+(q)$, $p=3$ and $k=8$. By~\eqref{NewEq} and~\eqref{EQ1} we have $r_p(\Aut(T))\le k+1=9$ and hence   
$$6r_p(\Aut(T))p^{3\ell r_p(\Aut(T))+1}\leq 6\cdot 9\cdot 3^{3\ell\cdot 9+1}.$$ 

By going through each family one-by-one and with some computation, one may use these inequalities to see that~\eqref{eq:eq1} holds, except  possibly when $T$ is one of the following groups:
\begin{align*}
&{}^2B_2(2^{3}) ,\, {}^2B_2(2^{5}) ,\, {}^2B_2(2^{7}) ,\, \mathrm{P}\Omega_8^+(2),\,\mathrm{P}\Omega_8^-(2),\,\mathrm{PSp}_6(2),\,\mathrm{PSp}_4(q)\,(\mathrm{with}\, q\in \{3,4,5,7\}),\\
&\mathrm{PSL}_2(q),\,\mathrm{PSL}_3(q),\,\mathrm{PSU}_3(q),\\
&\mathrm{PSL}_4(q)\,(\mathrm{with}\, q\in \{3,4,5,9\}),\,\mathrm{PSU}_4(q)\,(\mathrm{with}\, q\in \{2,3,4,5,7\}),\\
&\mathrm{PSL}_5(2),\,\mathrm{PSU}_5(2),\,\mathrm{PSL}_6(2),\,\mathrm{PSU}_6(2).
\end{align*}

We now study these exceptions in turn.
By factorizing the order of ${}^2B_2(2^3)$, ${}^2B_2(2^5)$ and ${}^2B_2(2^7)$ (and using $r_5(\Aut({}^2B_2(2^5)))=2$, which can be verified with \texttt{magma}), we see that~\eqref{eq:eq1} holds in these cases except when $(T,p,\ell)=({}^2B_2(2^3),13,1)$ which appears in Table~\ref{table:bt}.

Similarly, using the fact that $r_3(\Aut(\POmega_8^+(2)))=4$, $r_5(\Aut(\POmega_8^+(2)))=2$, $r_7(\Aut(\POmega_8^+(2)))=1$, $r_3(\Aut(\POmega_8^-(2)))=3$ and $r_p(\Aut(\POmega_8^-(2)))=1$ for $p\in \{5,7,17\}$, one can check that~\eqref{eq:eq1} holds when $T$ is $\POmega_8^+(2)$ or $\POmega_8^-(2)$. The case when $T=\PSp_6(2)$ or $T=\PSp_4(q)$ with  $q\in \{3,4,5,7\}$ is dealt with in a similar way.

Suppose now that $T=\PSL_n(q)$ (with $n\geq 2$) or $\PSU_n(q)$ (with $n\geq 3$). Since the groups $\PSL_2(q)$, $\PSL_3(q)$ and $\PSU_3(q)$ appear in Table~\ref{table:bt}, we may assume that $n\geq 4$. In particular, there are only finitely many exceptions left to deal with. For each of these, we compute the exact value of $r_p(\Aut(T))$ for each prime divisor of $p$ of $|\Aut(T)|$ with $p\neq r$. It is then straightforward to check that~\eqref{eq:eq1} holds in each of these cases except when $(T,p,\ell)=(\PSU_5(2),3,1)$, or $(T,p)=(\PSU_4(2),3)$ which, in view of the exceptional isomorphism $\PSU_4(2)\cong \PSp_4(3)$, was already dealt with in the case $r=p$.
\end{proof}

\begin{lemma}\label{Sylowp}Let $n\geq 1$, let $p$ be a prime and let $P$ be a Sylow $p$-subgroup of $\Sym(n)$. Then $|P|\leq p^{(n-1)/(p-1)}$. Moreover, the function $p\mapsto p^{1/(p-1)}$ is strictly decreasing.
\end{lemma}

\begin{proof}
Let $n=a_0+a_1p+\cdots +a_kp^k$ be the $p$-adic expansion of $n$, that is, $a_0,\ldots,a_k\in \{0,\ldots,p-1\}$ with $a_k\neq 0$. Let $P_i$ be a Sylow $p$-subgroup of $\Sym(p^i)$ and observe that $|P_i|=p^{\frac{p^i-1}{p-1}}$. It is well-known that 
$$P\cong \underbrace{P_0\times \cdots \times P_0}_{a_0\, \mathrm{times}}\times 
 \underbrace{P_1\times \cdots \times P_1}_{a_1\, \mathrm{times}}\times \cdots\times
 \underbrace{P_k\times \cdots \times P_k}_{a_k\, \mathrm{times}}
$$ from which it follows  that
\begin{equation*}\label{eqsym1}|P|=\prod_{i=0}^kp^{a_i\frac{p^i-1}{p-1}}.
\end{equation*}

We have
\begin{equation*}\label{eqsym2}\sum_{i=0}^ka_i\frac{p^i-1}{p-1}=\frac{1}{p-1}\left(\sum_{i=0}^ka_ip^i-\sum_{i=0}^ka_i\right)=\frac{1}{p-1}\left(n-\sum_{i=0}^ka_i\right)\leq \frac{n-1}{p-1}.
\end{equation*}

Finally, it is easy to verify that $x\mapsto x^{1/(x-1)}$ is a strictly decreasing function when $x>0$.
\end{proof}

An easy computation in modular arithmetic yields the following lemma:

\begin{lemma}\label{neardental3}Let $r$ be a prime, let $p$ be an odd prime and let $t$ be the smallest integer such that $p$ divides $r^t-1$. Let $p^e$ be the largest power of $p$ dividing $r^t-1$. Then $p^{e+g}$ divides $r^{f}-1$ if and only if $p^gt$ divides $f$. 
\end{lemma}

\begin{table}[!h]
\begin{center}
\begin{tabular}{|c|c|c|}\hline
$T$& $p$ &$\ell$\\\hline
$\Alt(5)$ &$5$&$2,3,5,6,7,10,25$\\
$\Alt(6)$ &$3$&$1,2,3$\\
$\Alt(9)$ &$3$&$1$\\\hline
$\PSL_2(3^3)$&$3$ &$\ell$\\
$\PSL_2(3^4),\PSL_2(3^6),\PSL_2(3^9)$&$3$ &$1$\\
$\PSL_2(5^5)$&$5$&$1$\\
$\PSL_2(p)$&$7,11,13$ &$2$\\
$\PSL_2(p^2)$&$5,7,11,13$ &$1$\\
$\PSL_3(p)$&$3, 5, 7, 11, 13, 19$    & $1$ \\
$\PSL_4(p)$&$3, 5, 7$    & $1$ \\
$\PSU_3(p)$&$3,5,7,11,17,23,29$    & $1$ \\
$\PSU_4(3)$&$3$ & $1,2$ \\
$\PSU_4(p)$&$5,7$    & $1$ \\
$\PSp_4(3)$&$3$ &$1,2,3,4,6,9$\\  
$\PSp_4(3^2),\PSp_4(3^3)$&$3$ &$1$\\  
$\PSp_4(p)$& $5,7,11,13,17,19,23,29$&$1$\\  
$\PSp_6(p)$&$3,5,7$&$1$\\
$G_2(3)$&$3$&$1,2$\\
$G_2(9)$&$3$&$1$\\\hline
$\PSL_2(2^3)$&$3$&$\ell$\\\hline
\end{tabular}
\end{center}
\caption{Exceptional cases in Theorem~\ref{secondStrategy}}\label{tb:final}
\end{table}

\begin{theorem}\label{secondStrategy}
Let $p$ be an odd prime, let $\ell\geq 1$ and let $T$ be an non-abelian simple group such that $(T,p,\ell)$ appears in Table~$\ref{table:bt}$. Let $P$ be a Sylow $p$-subgroup of $\Aut(T)\wr\Sym(\ell)$. Then either $P$ is cyclic, or $\ell|T|^\ell>4p|P|^2\log_p|P|$, or $(T,p,\ell)$ appears in Table~$\ref{tb:final}$.
\end{theorem}
\begin{proof}
We assume that $P$ is not cyclic. Let $Q$ be a Sylow $p$-subgroup of $\Aut(T)$. From Lemma~\ref{Sylowp}, we have $|P|\leq p^{(\ell-1)/(p-1)}|Q|^\ell$ and, since $p^{1/(p-1)}$ is a decreasing function of $p$, satisfying the inequality 
\begin{equation}\label{newneweq}
\ell|T|^\ell>4p|Q|^{2\ell}3^{\ell-1}\log_3(|Q|^\ell3^{(\ell-1)/2})
\end{equation} 
is sufficient to satisfy the inequality $\ell|T|^\ell>4p|P|^2\log_p|P|$. The result then follows by considering the triples $(T,p,\ell)$ that appear in Table~\ref{table:bt} one-by-one. This is tedious but not very difficult: most of them can be eliminated using~\eqref{newneweq}. (Knowledge of $|\Aut(T)|$ is also required, see~\cite{ATLAS} for example).  The only cases that require more care are when $T$ is isomorphic to one of $\PSL_2(r^f)$, $\PSL_3(r^f)$ or $\PSU_3(r^f)$ and $r\neq p$. We discuss these in more detail.

Let $T_0$ be the group of inner-diagonal automorphisms of $T$. First suppose that $p$ is coprime to $|T_0|$. Then $f\geq|Q|$, $f\geq p$ and thus
$$4f^{2\ell+1}3^{\ell-1}\log_3(f^{\ell}3^{(\ell-1)/2})\geq 4p|Q|^{2\ell}3^{\ell-1}\log_3(|Q|^\ell3^{(\ell-1)/2}).$$
With some computation, one can see that $\ell|T|^\ell >4f^{2\ell+1}3^{\ell-1}\log_3(f^{\ell}3^{(\ell-1)/2})$ and thus \eqref{newneweq} is satisfied. 

We now assume that $p$ divides $|T_0|$. Suppose that $T=\PSL_2(r^f)$, let $t$ be the smallest integer with $p\mid r^t-1$ and let $p^e$ be the largest power of $p$ dividing $r^t-1$. Since $p$ divides $|T_0|=r^f(r^{2f}-1)$ and $p\neq r$, we get $p\mid (r^{2f}-1)$ and hence $t\mid {2f}$. Similarly, since $p\mid r^{p-1}-1$, we have $t\mid p-1$ and hence $t$ is coprime to $p$. 

We first consider the case that $t\notin\{f,2f\}$. Since $t$ is a divisor of $2f$, we must have $t\leq 2f/3$. Write $2f=tp^gk$, with $k$ coprime to $p$. Observe that $p^g$ is the largest power of $p$ dividing $f$. By Lemma~\ref{neardental3}, $p^{e+g}$ is the largest power of $p$ dividing $r^{2f}-1$. Since $|\Aut(T)|=f\cdot r^f(r^{2f}-1)$, it follows that $|Q|=p^{e+2g}$ and hence 
\begin{equation}\label{EQQ1}
|Q|\leq p^ef^2\leq(r^t-1)f^2\leq (r^{2f/3}-1)f^2.
\end{equation} 
Note that, for $r$ odd, we have $p\leq (r^t-1)/2\leq (r^{2f/3}-1)/2$. Using this bound when $r$ is odd and \eqref{EQQ1}, a computation shows that \eqref{newneweq} holds, except when $r=2$ and $f\leq 24$, $r=3$ and $f\leq 9$, or $r=5$ and $f\leq 3$. These values of $r$ and $f$ can then be checked one-by-one.

Finally, suppose that $t\in \{f,2f\}$, that is, $p$ is a primitive prime divisor of $r^f-1$ or of $r^{2f}-1$. Since $t$ is coprime to $p$, so is $f$ and hence $Q$ is a subgroup of $T_0$. As $|T_0:T|=2$ and $p$ is odd, we get $Q\leq T$. Now it follows from the subgroup structure of $\PSL_2(r^f)$ (see~\cite[Theorem~$6.25$]{Suzuki} for example) that $Q$ is cyclic and $|Q|$ divides $q-1$ when $t=f$ and divides $q+1$ when $t=2f$. Since $P$ is not cyclic, we have $\ell\geq 2$. Moreover, $|Q|=(q\pm 1)/c$ for some $c\geq 1$.   Using this explicit value of $|Q|$, it is then a simple case-by-case analysis to check that the result holds.

The cases when $T$ is one of $\PSL_3(r^f)$ or  $\PSU_3(r^f)$ can be handled similarly.
\end{proof}

We are now ready to prove the main theorem of this section.

\begin{proof}[Proof of Theorem~$\ref{theo:semisimple}$]
Since $N$ is a minimal normal subgroup of $G$, we can write $N=T^\ell$ for some non-abelian simple group $T$ and some $\ell\geq 1$. Moreover, $G$ is a subgroup of the wreath product $W=\Aut(T)\wr\Sym(\ell)$ that acts transitively on the $\ell$ simple direct factors of $N$ and hence $\ell |T|^\ell\leq |G|$.  

\medskip
\textbf{The case when $p$ is odd.}
Assume that $|\V(\Gamma)|\leq 2p|G_v^*|\log_p|G_v^*|$. Observe that $|\V(\Gamma)|=|G:G_v|=|G|/(\chi|G_v^*|)$ and hence 
\begin{equation}\label{QE0}
|G|\leq 2p\chi |G_v^*|^2\log_p|G_v^*|.
\end{equation}
By Lemma~\ref{lemma:ea}~(\ref{bonbon}), we have $ |G_v^*|^{2/3}\leq p^{r_p(G_v^*)}$. Since $r_p(G_v^*)\leq r_p(W)$ and $\ell |T|^\ell\leq |G|$, it follows 
\begin{align*}
\ell|T|^\ell\leq |G|&\leq  2p\chi|G_v^*|^2\log_2|G_v^*|\leq 2p\chi p^{3r_p(G_v^*)}\left(\frac{3}{2}r_p(G_v^*)\right)\\
&\leq 3\chi p^{3r_p(W)+1}r_p(W)\leq 6 p^{3r_p(W)+1}r_p(W).
\end{align*}
In particular, by Theorem~\ref{thrm:ss1}, $(T,p,\ell)$ appears in Table~\ref{table:bt}. Let $P$ be a Sylow $p$-subgroup of $W$ containing $G_v^*$. Since $(G_v^*)^{\Gamma(v)}$ is an elementary abelian $p$-group of order $p^2$, the group $G_v^*$ is not cyclic and hence neither is $P$. Moreover
\begin{equation*}
\ell |T|^\ell\leq |G|\leq 2p\chi |G_v^*|^2\log_p|G_v^*|\leq 4p|P|^2\log_p|P|.
\end{equation*}
In particular, by Theorem~\ref{secondStrategy}, $(T,p,\ell)$ appears in Table~\ref{tb:final}. Recall that from  Theorem~\ref{lemma:ea}~(\ref{bonbonbon}) the group $G_v^*$ has nilpotency class at most $3$.

We first deal with the only two infinite families in this table. Suppose that $T$ is  isomorphic to either $\PSL_2(2^3)$ or $\PSL_2(3^3)$ and thus $p=3$. We will show that $\ell\leq 2$ when $T=\PSL_2(3^3)$ and $\ell\leq 4$ when $T=\PSL_2(2^3)$.

Let $B=\Aut(T)^\ell$ and let $a,b,c$ be natural numbers such that $|G_v^*:G_v^*\cap B|=3^a$, $|G_v^*\cap B:G_v^*\cap N|=3^b$ and $|G_v^*\cap N|=3^c$. Moreover let $\ell'$ be the largest divisor of $\ell$ coprime to $3$. Observe that $|G_v^*|=3^{a+b+c}$, $|G:G\cap B|\geq \ell'3^a$, $|G\cap B: N|\geq 3^b$ and thus $\ell' 3^{a+b}|T|^\ell\leq |G:G\cap B||G\cap B:N||N|=|G|$. Combining this with~\eqref{QE0} we obtain 
\begin{equation}\label{EX1}
\ell'|T|^\ell\leq 12\cdot 3^{a+b+2c}(a+b+c).
\end{equation}
Note that $a\leq (\ell-1)/2$ by Lemma~\ref{Sylowp}. 

We first consider the case when $T=\PSL_2(3^3)$.  Thus $c\leq 3\ell$ and, as $B/N\cong\Out(T)^\ell$ is isomorphic to $\C_6^\ell$, we have $b\leq \ell$. Therefore \eqref{EX1} implies 
$$\ell'|T|^\ell\leq 12\cdot 3^{7\ell+(\ell-1)/2}(4\ell+(\ell-1)/2).$$
This yields $\ell\leq 3$. If $\ell=3$ then~\eqref{EX1} holds only if $a+b+2c\geq 21$, from which it follows that  $|P:G_v^*|\leq 3$. It can be checked with~\texttt{magma} that the only subgroup of $P$ having nilpotency class at most $3$ and index at most $3$ is $P\cap B$. Therefore $G_v^*=P\cap B$, $(a,b,c)=(0,3,9)$, $|G_v^*|=3^{12}$ and $|G|= |G:B\cap G||B\cap G|\geq \ell|B\cap G|\geq \ell |G_v^{*}N|=3\cdot 3^{3}|T|^3$. Now, it is easily seen that~\eqref{QE0} is not satisfied. Thus $\ell\leq 2$.

We now consider the case when $T=\PSL_2(2^3)$. Let $t$ be the highest power of $3$ dividing $(\ell!)$. Then $a\leq t$, $b\leq\ell$ and $c\leq 2\ell$. Therefore \eqref{EX1} implies 
$$\ell'|T|^\ell\leq 12\cdot 3^{5\ell+t}(3\ell+t).$$ 
This yields $\ell\leq 12$ or $\ell\in \{15,18,27\}$. (To see this, use the bound $t\leq (\ell-1)/2$ when $\ell\geq 27$ and use the explicit value of $t$ when $\ell\leq 27$.)

If $\ell\in \{5,7,8,10,11,15\}$ then a direct computation shows that, in each case, ~\eqref{EX1} gives $a+b+2c=5\ell+t$ and hence $P=G_v^*$. Since a Sylow $3$-subgroup of $\PGammaL_2(2^3)\wr\Sym(3)$ has nilpotency class $6$, $G_v^*$ has nilpotency class at least $6$, which is a contradiction. 

If $\ell\in \{12,18,27\}$ then~\eqref{EX1} yields $a+b+2c\geq 5\ell+t-1$ and hence $|P:G_v^*|\leq 3$. Let $R$ be a Sylow $3$-subgroup of $\PGammaL_2(2^3)\wr\Sym(3)$. A computation with \texttt{magma} gives that the only subgroup of $R$ having nilpotency class at most $3$ and index at most $3$ is $R\cap \PGammaL_2(2^3)^3$. As $\ell\geq 3$, it easily follows that $|P:G_v^*|=3$ and $G_v^*=P\cap B$. We have $3=|P:G_v^*|=|P:P\cap B|=3^t\geq 3^5$, a contradiction. 

If $\ell=6$ then~\eqref{EX1} implies $a+b+2c\geq 5\ell+t-2$. This yields  $|P:G_v^*|\leq 9$. With \texttt{magma} we obtain that the only subgroup of $P$  having nilpotency class at most $3$ and index at most $9$ is $P\cap B$. Thus $G_v^*=P\cap B$, $|G_v^*|=3^{18}$ and $G_v^*N=B$. Therefore  $|G|= |G:B||B|\geq \ell|B|=6(3|T|)^6$. Now, it is easily seen that~\eqref{QE0} is not satisfied.

If $\ell=9$ then~\eqref{EX1} implies $a+b+2c\geq 5\ell+t-3$. Thus $|P:G_v^*|\leq 27$. Assume that $G_v^*$ acts intransitively on the nine simple direct factors of $N$. Then $G_v^*\leq R\times R\times R$, where $R$ is a Sylow $3$-subgroup of $\PGammaL_2(2^3)\wr\Sym(3)$. Since $R$ has nilpotency class $6$, the group $G_v^*$ projects to a proper subgroup of $R$ in each of the three coordinates. Therefore $|G_v^*|\leq (|R|/3)^3=3^{27}$, contradicting the fact that $|P:G_v^*|\leq 27$. Thus $G_v^*$ acts transitively on the nine simple direct factors of $N$. Now we compute all the subgroups of $Q$ of $P$ with $|P:Q|\leq 27$ and with $Q$ projecting to a transitive subgroup of $\Sym(9)$ (observe that, computationally, the second requirement gives a strong constraint), and we check that they all have nilpotency class greater than $9$.

 We have shown that, if $T=\PSL_2(3^3)$ then $\ell\leq 2$ and if $T=\PSL_2(2^3)$ then $\ell \leq 4$. In particular, there are only a finite number of small cases in Table~\ref{tb:final} left to consider. We use \texttt{magma} to deal with them in the following way: for every triple $(T,p,\ell)$, we first determine all subgroups $G$ of $\Aut(T)\wr\Sym(\ell)$ projecting to a transitive subgroup of $\Sym(\ell)$. Second, for each such $G$, we determine the $p$-subgroups $Q$ of $G$ with  $|G|\leq 2p\chi |Q|^2\log_p|Q|$ (in most cases, there is no such $Q$ or $Q$ is a Sylow $p$-subgroup of $G$). Then, for each such $Q$, we check whether $|\Zent Q|^3\geq |Q|$, $Q$ has nilpotency class at most $3$ and exponent at most $p^2$ (see Lemma~\ref{lemma:ea}). For the groups $Q$ satisfying these criteria, we determine all the subgroups $H$ of $G$ with $Q\leq H\leq G$ and $|H:Q|=\chi$ and we compute the permutation representation $\hat{G}$ of $G$ on the right cosets of $H$. Finally, we check whether $\hat{G}$ acts on a connected graph of valency $2p$ by studying the suborbits of $\hat{G}$ of size $p$ and $2p$. The cases which do occur are listed in Table~\ref{mutherfucker}.

The procedure just described generally works well, but a few cases are more challenging computationally. These can be handled with a few tricks. For example, we show how to deal with the cases when $N=T^\ell$ is one of  $G_2(3)^2$ or $G_2(9)$. 

Suppose that $N=T=G_2(9)$.  A computation shows that~\eqref{QE0} is satisfied only when $G=T$, $G_v^*$ is a Sylow $3$-subgroup of $G$ and $\chi=2$. As above, let $Q$ be a  a Sylow $3$-subgroup of $T$. This is our candidate for $G_v^*$. For every maximal subgroup $K$ of $Q$ ($K$ is our candidate for $G_{vw}$ for $w\in \Gamma(v)$), we find  that $\norm G K\leq \norm G Q$ and hence $\norm G Q=\langle \norm G Q,\norm G K\rangle$, contradicting the fact that $\Gamma$ is connected and $G$-arc-transitive.

Suppose now that $N=G_2(3)^2$. A computation shows that~\eqref{QE0} is satisfied only when $|G:N|=2$ $G_v^*$ is a Sylow $3$-subgroup of $G$ and $\chi=2$. As $W/N\cong \D_4$, $W$ contains a unique conjugacy class of subgroups $G$ with $|G:N|=2$ and with $G$ acting transitively on $\{T_1,T_2\}$. Thus $G=(T\times T)\rtimes\langle \iota\rangle$, where $\iota$ is the involutory automorphism of $T\times T$ defined by $(x,y)^\iota=(y,x)$. As above, let $Q$ be a Sylow $3$-subgroup of $G$ and, for every maximal subgroup $K$ of $Q$, we find that $\langle \norm G Q,\norm G K\rangle<G$, contradicting the fact that $\Gamma$ is connected and $G$-arc-transitive.

\medskip
\textbf{The case when $p=2$.} If $\chi=2$, then the proof follows immediately from~\cite[Theorem~$1.2$ and Table~$2$]{PSV4valent}. The relevant examples are reported in Table~\ref{mutherfucker}. 

We may thus assume that $\chi=1$. In particular $G_v^*=G_v$ is a $2$-group.   Assume that $|\V(\Gamma)|\leq 4|G_v|\log_2|G_v|$. Let $S$ be a Sylow $2$-subgroup of $G$ with $G_v\leq S$. Write $\ell=2^{\ell_e}\ell_o$ with $\ell_o$ odd, and $|T|=2^to$ with $o$ odd.  Now, $|\V(\Gamma)|=|G:G_v|\geq |G:S|\geq \ell_o o^{\ell}$. Thus $\ell_oo^\ell \leq 4|G_v|\log_2|G_v|$. From Lemma~\ref{lemma:ea}~(\ref{bonbon}) and Lemma~\ref{neandertal1}, we have $|G_v|\leq (2^{r_2(G)})^{3/2}\leq 2^{3r_2(W)/2}=2^{3\ell r_2(\Aut(T))/2}$. Therefore
$$\ell_o o^\ell\leq 4|G_v|\log_2|G_v|\leq 6\ell r_2(\Aut(T))2^{3\ell r_2(\Aut(T))/2}.$$
By~\cite[Theorem~$4.4$]{PSV4valent}, one of the following holds:
\begin{itemize}
\item $\ell\in \{1,2,3,4\}$ and $T=\Alt(5)$ or $\Alt(6)$;
\item $\ell\in \{1,2,4\}$ and $T=\PSL_2(8)$ or $\PSL_3(2)$;
\item $\ell\in \{1,2\}$ and $T=\PSL_2(2^f)$ (with $f\in \{4,5,6\}$), $\Alt(8)$, $\PSL_3(4)$, $\Sp_4(4)$, $\PSp_4(3)$ or $\PSp_6(2)$;
\item $\ell=1$ and $T=\PSL_2(2^f)$ (with $f\in \{7,\ldots,12\}$), $M_{12}$, $M_{22}$, $\Alt(7)$, $\PSL_2(11)$, $\PSL_2(13)$, $\PSL_2(25)$, $\PSL_5(2)$, $\PSL_6(2)$, $\PSL_3(3)$, $\Sp_4(8)$, $\Sp_4(16)$, $\Sp_4(32)$, $\Sp_6(4)$, $\Sp_8(2)$, $\Omega_8^+(2)$, $\PSU_3(3)$, $\PSU_6(2)$ or $\Omega_8^-(2)$.
\end{itemize}
In particular, we have only a finite number of relatively small groups to consider. We deal with these case-by-case with the help of \texttt{magma} in a very much similar way as for $p>2$. 

For each pair $(T,\ell)$, we first determine all subgroups $G$ of $\Aut(T)\wr\Sym(\ell)$ projecting to a transitive subgroup of $\Sym(\ell)$. Second, for each such $G$, we determine the $2$-subgroups $Q$ of $G$ with  $|G|\leq 4|Q|^2\log_2|Q|$ (in most cases, there is no such $Q$ or $Q$ is a Sylow $2$-subgroup of $G$). Then, for each such $Q$, we check whether $|\Zent Q|^3\geq |Q|$, $Q$ has nilpotency class at most $3$ and exponent at most $4$ (see Lemma~\ref{lemma:ea}). For the groups $Q$ satisfying these criteria, we compute the permutation representation $\hat{G}$ of $G$ on the right cosets of $Q$. Finally, we check whether $\hat{G}$ acts on a connected asymmetric arc-transitive digraph of out-valency $2$ by studying the self-paired suborbits of $\hat{G}$ of size $2$. The cases which do occur are listed in Table~\ref{mutherfucker}.

The procedure just described generally works well, but a few cases are more challenging computationally. These can be handled with a few tricks. For example, we show how to deal with the cases when $T$ is isomorphic to  $\Sp_4(16)$ or $\Sp_6(4)$. Observe that, in these cases, $\ell=1$. When $T=\Sp_6(4)$, it is easily checked that $T\leq G\leq \Aut(T)$ and $|G|\leq 4|Q|^2\log_2|Q|$ imply that $Q$ is a Sylow $2$-subgroup of $\Aut(T)$.  However a Sylow $2$-subgroup of $\Sp_6(4)$ has nilpotency class $4$ and hence $Q$ has nilpotency class at least $4$, a contradiction. 

Finally, suppose that $T=\Sp_4(16)$. Applying the same method as above yields that $|\Aut(T):G|\leq 2$ and $Q$ is a Sylow $2$-subgroup of $G$. Let $F=\langle\sigma\rangle$ be the group of field automorphisms of $T$. Clearly $|\Aut(T):(T\rtimes F)|=2$. Moreover, \texttt{magma} has the good taste to have $T\rtimes F$ in its library and hence it is a computation to see that a Sylow $2$-subgroup of $T\rtimes F$ has nilpotency class $8$. This yields that $G\neq \Aut(T)$ and $G\neq T\rtimes F$, and hence $G\cap (T\rtimes F)=T\rtimes \langle\sigma^2\rangle$. Another computation in \texttt{magma} gives that a Sylow $2$-subgroup of $T\rtimes \langle \sigma^2\rangle$ has nilpotency class $4$.
\end{proof}

\section{The case $(p,\chi)=(2,1)$.}\label{sec:(2,1)}

In this section, we obtain the complete list of pairs $(\Gamma,G)$ that are exceptional with respect to Theorem~\ref{theorem:main2} in the case $(p,\chi)=(2,1)$. These can be found in Table~\ref{tb:exceptions}.

\begin{table}[h]
\begin{center}
\begin{tabular}{|c|c|c|c|}\hline
$\Gamma$ &$|\V\Gamma|$&$|G_v|$&$G$\\\hline
$\AG(F_6)$& $18$&$4$&$\C_3^2\rtimes \D_4$\\
$\AG(\Pet)$&$30$&$4$&$\Sym(5)$\\
$\AG(\Hea)$&$42$&$8$&$\PGL_2(7)$\\
$\C^{\pm 1}(3,3,3)$&$81$&$8$&$(\C_3^3\rtimes \C_2)\rtimes \Alt(4)$\\
$\HC(\Hea)$&$84$&$8$&$\PGL_2(7)\times\C_2$\\
$\AG(\Tut)$&$90$&$8$ or $16$&$|\PGammaL_2(9):G|\leq 2$\\
$\HC(\Tut)$&$180$&$16$&$\Sym(6)\rtimes\C_4$\\
$\AAG(\Tut)$&$8100$&$256$&$|\PGammaL_2(9)\wr \Sym(2):G|=2$\\\hline
\end{tabular}
\caption{Exceptional pairs when $(p,\chi)=(2,1)$. \footnotesize (For the notation see~\cite[Section 2.2]{PSV4valent}.)}\label{tb:exceptions}
\end{center}
\end{table}

Throughout this section, we will work under Hypothesis~A with $(p,\chi)=(2,1)$. In particular, since $\chi=1$, we have $G_v=G_v^*$.

\begin{theorem}\label{theorem : (2,1)}
Let $(\Gamma,G)$ be a locally-$\L_{2,1}$ pair such that $\Gamma$ is connected and $G$-edge-transitive and let $v$ be a vertex of $\Gamma$. Then one of the following occurs:
\begin{enumerate}
\item $\Gamma\cong \PX(2,r,s)$ for some $r\geq 3$ and $1\leq s\leq r-2$; \label{haha}
\item $|\V(\Gamma)|\geq 4|G_v|\log_2 |G_v|$;  \label{hehe}
\item $|G_v|\leq 16$; \label{hoho}
\item $(\Gamma,G)$ appears in the last line of Table~\ref{tb:exceptions}. \label{hihi}
\end{enumerate}
\end{theorem}
\begin{proof}
We argue by induction on $|\V(\Gamma)|$ and apply Theorem~\ref{theorem:main1} with $(p,\chi)=(2,1)$. If Theorem~\ref{theorem:main1}~(\ref{PXU}) or~(\ref{bound}) holds then~(\ref{haha}) or~(\ref{hehe}) holds.  We now prove the following claim from which the result will follow.

\smallskip

\noindent\textsc{Claim: }Suppose that Theorem~\ref{theorem:main1}~(\ref{twoorbits}),~(\ref{cycle}) or~(\ref{semisimple}) is satisfied. Then either~(\ref{hehe}),~(\ref{hoho}) or~(\ref{hihi}) holds.

\smallskip 

If Theorem~\ref{theorem:main1}~(\ref{twoorbits}) is satisfied then $G$ has a semiregular abelian minimal normal subgroup having at most two orbits and, by Theorem~\ref{thm:mainbasic}, it follows that $|G_v|=4$. 

If Theorem~\ref{theorem:main1}~(\ref{cycle}) is satisfied then $G$ has a semiregular abelian minimal normal subgroup $N$ such that $\Gamma/N$ is a cycle of length $m$ at least $3$ and~(\ref{haha}) does not hold. Clearly, $N$ is a $q$-group for some prime $q$. Let $|G_v|=2^t$ and thus $4|G_v|\log_2( |G_v|)=t2^{t+2}$. By Theorem~\ref{GarPraCycle}, $q$ is odd and $|\V(\Gamma)|\geq tq^{t}$. If $tq^{t}\geq t2^{t+2}$ then~(\ref{hehe}) holds. We may thus assume that $tq^{t}< t2^{t+2}$ which easily implies that $t\leq 3$ and hence~(\ref{hoho}) holds.

Finally, if Theorem~\ref{theorem:main1}~(\ref{semisimple}) is satisfied then $G$ has a unique minimal normal subgroup and this subgroup is non-abelian and the claim follows  from Theorem~\ref{theo:semisimple}.~$_\blacksquare$ 

\smallskip

Finally we assume that Theorem~\ref{theorem:main1}~(\ref{funny}) holds. Thus $G$ has a non-identity normal subgroup $N$ such that $\Gamma$ is a regular cover of $\Gamma/N$ and $(\Gamma/N,G/N)$ satisfies  Theorem~\ref{theorem:main1}~(\ref{twoorbits}), (\ref{cycle}), or (\ref{semisimple}), with $(\Gamma,G)$ replaced by $(\Gamma/N,G/N)$.  Since $N$ is nontrivial, by induction we have that $(\Gamma/N,G/N)$ satisfies the conclusion of this theorem. 

If $(\Gamma/N,G/N)$ satisfies~(\ref{hehe}) or~(\ref{hoho}) then, since $\Gamma$ is a regular cover of $\Gamma/N$, so does $(\Gamma,G)$. If $(\Gamma/N,G/N)$ satisfies~(\ref{hihi}) then $|\V(\Gamma/N)|=8100$ and $|G_v|=|G_vN/N|=256$. Now $4|G_v|\log_2|G_v|=8192<2\cdot |\V(\Gamma/N)|\leq |N||\V(\Gamma/N)|$ and~(\ref{hehe}) holds for $(\Gamma,G)$. 

By our claim, these are the only possibilities.
\end{proof}

If $(\Gamma,G)$ is a pair satisfying the hypothesis of Theorem~\ref{theorem : (2,1)} but neither part~(\ref{haha}),~(\ref{hehe}) or ~(\ref{hihi}) of the conclusion then $|G_v|\leq 16$ and thus $|\V(\Gamma)|<256$. In particular, such a pair must appear in the census obtained in \cite{PSVHat}. It is then simply a matter of going through this census to obtain the following corollary to Theorem~\ref{theorem : (2,1)}.

\begin{corollary}\label{cor : (2,1)}
Let $(\Gamma,G)$ be a locally-$\L_{2,1}$ pair such that $\Gamma$ is connected and $G$-edge-transitive and let $v$ be a vertex of $\Gamma$. Then one of the following occurs:
\begin{enumerate}
\item $\Gamma\cong \PX(2,r,s)$ for some $r\geq 3$ and $1\leq s\leq r-2$;
\item $|\V(\Gamma)|\geq 4|G_v|\log_2( |G_v|)$; 
\item $(\Gamma,G)$ appears in Table~\ref{tb:exceptions}. \label{exceptions}
\end{enumerate}
\end{corollary}

\end{document}